\documentclass[a4paper,11pt,oneside]{amsart}
\usepackage[english]{babel} %permet de gérer les règles typographiques spécifiques. english pour l'anglais, french pour le français (guillemets avec \og, \fg etc).
\usepackage[T1]{fontenc} %prise en charge en sortie des mêmes caractères
\usepackage[utf8x]{inputenc} %prise en charge en entrée des caractéres accentués comme é, ç, etc

\usepackage{amssymb,amsmath,amsfonts,amsxtra,amsthm} %pour les symboles mathématiques
\usepackage{stmaryrd} %symboles supplémentaires à ceux compris dans amssymb
\usepackage{mathtools} % possibilités supplémentaires pour taper des maths que celles données par les packages standards de l'ams
\usepackage{mathrsfs} %permet d'avoir une d'élégantes lettres majuscules allongées en tapant : \mathscr{quelquechose}
\usepackage{cases} %pour pouvoir utiliser les environnements cases et numcases, faire des systèmes et numéroter chaque équation séparément, 
\usepackage[all]{xy} %pour faire des diagrammes commutatifs

\usepackage{geometry} %pour gérer les marges
\usepackage{lmodern} %pour que les polices soient de meilleure qualité que celles par défaut, permet d'avoir des pdf de qualité vectorielle
\usepackage{xargs} %pour pouvoir faire des macros avec arguments optionnels
\usepackage{graphicx} %Pour inclure des figures via \includegraphics[scale=•]{•}
\PassOptionsToPackage{svgnames}{xcolor}
\usepackage[svgnames,table]{xcolor} %Pour pouvoir mettre plus de couleur. 

\usepackage{fancyhdr}
\usepackage{etoolbox}
\usepackage{overpic}
\usepackage{bm}

\usepackage[
  style=alphabetic,
  backend=bibtex,
  maxbibnames=10,       % En maths, on liste souvent tous les auteurs
  maxcitenames=2,
  giveninits=true,
  terseinits=true,
  doi=true,
  url=false,
  isbn=false,
  eprint=true,           % Utile pour arXiv en maths
  hyperref = true
]{biblatex}
\renewbibmacro{in:}{%
  \ifentrytype{article}{}{\printtext{\bibstring{in}\intitlepunct}}}
\DeclareFieldFormat[article,inbook,incollection,inproceedings,patent,thesis,unpublished]{title}{#1}
\DeclareFieldFormat[article,inbook,incollection,inproceedings,patent,thesis,unpublished]{citetitle}{#1}
\AtEveryBibitem{
  \clearfield{month}
  \clearfield{note}
  \clearlist{language}
}
\usepackage[colorlinks=true,citecolor=Magenta,linkcolor=Blue, urlcolor=Green, pdftoolbar=true,
bookmarksnumbered=true,
  bookmarksopen=true,
  bookmarksopenlevel=2,
  pdfstartview=Fit]{hyperref} %pour avoir des liens hypertextes dans le pdf. %Important apparemment qu'il soit en dernier. hidelinks pour cacher tous les liens hypertextes.
\hypersetup{bookmarks=true}
\usepackage{bookmark}
\usepackage{subfig}

\newcommand{\ensemblenombre}[1]{\ensuremath{\mathbb{#1}}}
\newcommand{\N}{\ensemblenombre{N}}
\newcommand{\Z}{\ensemblenombre{Z}}

\newcommand{\R}{\ensemblenombre{R}}

\newcommand{\abs}[1]{\ensuremath{\left\lvert#1\right\rvert}}

\newcommand{\enstq}[2]{\ensuremath{\left\{#1\mathrel{}\middle|\mathrel{}#2\right\}}}
\newcommand{\prodscal}[2]{\ensuremath{\mathopen{\langle}#1\mathclose{}\mathpunct{},#2\mathclose{\rangle}}}

\newcommand{\intervalle}[4]{\ensuremath{\mathopen{#1}#2
		\mathclose{}\mathpunct{};#3
		\mathclose{#4}}}
\newcommand{\intervalleff}[2]{\intervalle{[}{#1}{#2}{]}}

\newcommand{\intervallefo}[2]{\intervalle{[}{#1}{#2}{[}}
\newcommand{\intervalleoo}[2]{\intervalle{]}{#1}{#2}{[}}

\newcommand{\restreinta}{\ensuremath{\mathclose{}|\mathopen{}}}
\newcommand{\cc}{\mathcal{C}}
\newcommand{\piun}[1]{\pi_1({#1})}

\newcommand{\RP}[1]{\ensuremath{\R\mathbf{P}^{#1}}}
\newcommand{\Sn}[1]{\mathbf{S}^{#1}}

\newcommand{\Hn}[1]{\mathbf{H}^{#1}}

\newcommand{\Tn}[1]{\mathbf{T}^{#1}}

\newcommand{\Tan}[2]{\ensuremath{\mathrm{T}_{#1}#2}}

\newcommand{\dSancien}{\ensuremath{\mathrm{dS}^2}}

\newcommand{\dS}{\ensuremath{\mathbf{dS}^2}}
\newcommand{\dStilde}{\ensuremath{\widetilde{\mathbf{dS}}^2}}

\newcommand{\dSsingularthetatilde}{\ensuremath{\widetilde{\mathbf{dS}}^2_{\theta}}}

\newcommand{\mudS}{\ensuremath{\bm{g}}}

\newcommand{\Falpha}{\ensuremath{\mathcal{F}_\alpha}}
\newcommand{\Fbeta}{\ensuremath{\mathcal{F}_\beta}}

\newcommand{\odS}{\ensuremath{\mathsf{o}}}
\newcommand{\odSsingulartheta}{\ensuremath{\mathsf{o}_\theta}}

\newcommand{\azero}{\ensuremath{\mathsf{a}}}
\newcommand{\bzero}{\ensuremath{\mathsf{b}}}
\newcommand{\m}{\ensuremath{g}}
\newcommand{\mzero}{\ensuremath{\mathsf{g}}}
\newcommand{\zero}{\ensuremath{\mathsf{0}}}

\newcommand{\SL}[1]{\ensuremath{\mathrm{SL}_{#1}\mathopen{(}\R\mathclose{)}}}

\newcommand{\PSL}[1]{\ensuremath{\mathrm{PSL}_{#1}\mathopen{(}\R\mathclose{)}}}

\newcommand{\Bisozero}[2]{\ensuremath{\mathrm{SO}^0\mathopen{(}#1\mathpunct,#2\mathclose{)}}}

\newcommand{\Deftheta}{\ensuremath{\mathsf{Def}_{\theta}(\Tn{2},\mathsf{0})}}
\newcommand{\orientation}[2]{\ensuremath{\mathsf{or}\mathopen{(}#1\mathpunct,#2\mathclose{)}}}
\newcommand{\angleLorentzien}[2]{\ensuremath{\left(\mkern-3mu\left(#1\mathpunct,#2\right)\mkern-3mu\right)}}

\newcommand{\Homologie}[3]{\ensuremath{\mathrm{H}_{#1}\mathopen{(}#2\mathpunct,#3\mathclose{)}}}

\DeclareMathOperator{\Int}{Int}

\DeclareMathOperator{\id}{id}

\DeclareMathOperator{\Stab}{Stab}

\DeclareMathOperator{\Homeo}{Homeo}

\DeclareMathOperator{\Cl}{Cl}

\DeclareMathOperator{\PMod}{PMod}

\DeclareMathOperator{\Conv}{Conv}
\DeclareMathOperator{\hol}{hol}

\DeclareMathOperator{\arcosh}{arcosh}

\numberwithin{equation}{section}
\numberwithin{figure}{section}

\AtBeginDocument{}

\theoremstyle{definition}
\newtheorem{definition}{Definition}[section]

\theoremstyle{plain}
\newtheorem{theorem}[definition]{Theorem}

\newtheorem{theoremintro}{Theorem}

\newtheorem{propositionintro}[theoremintro]{Proposition}

\newtheorem{corollary}[definition]{Corollary}

\newtheorem{lemma}[definition]{Lemma}
\newtheorem{proposition}[definition]{Proposition}

\theoremstyle{remark}

\newtheorem{remark}[definition]{Remark}
\newtheorem{remarks}[definition]{Remarks}

\title[Rigidity of the timelike marked
length spectrum]{Rigidity of the timelike marked
length spectrum
and length-twist coordinates
of singular de-Sitter tori}
\author{Martin Mion-Mouton}
\date{\today}

\subjclass[2020]{53C50,57M50,53C24,53C22,51M50}
\keywords{Lorentzian surfaces, Singular de-Sitter tori, Geodesics}

\definecolor{Cerulean}{RGB}{0,123,167}
\definecolor{CornflowerBlue}{RGB}{100,149,237}
\definecolor{NavyBlue}{RGB}{0,0,128}
\definecolor{DodgerBlue}{RGB}{30,144,255}
\definecolor{Fuchsia}{RGB}{255,0,255}
\definecolor{MediumTurquoise}{RGB}{72,209,204}

\begin{document}
\address{
Martin Mion-Mouton,
Institut de Mathématiques de Marseille (I2M)
}
\email{\href{mailto:martin.mion-mouton@univ-amu.fr}{martin.mion-mouton@univ-amu.fr}}
\urladdr{\url{https://www.i2m.univ-amu.fr/perso/martin.mion-mouton/index.html}}

\begin{abstract}
In this paper,
we study
the closed timelike geodesics of
de-Sitter tori with one singularity
and prove their uniqueness
in their free homotopy class.
We introduce the notion of timelike marked length spectrum of such a torus, and establish its rigidity with respect to the lengths of two homotopy classes of intersection number one.
We also construct length-twist coordinates on the deformation space
of de-Sitter tori with one singularity.
\end{abstract}

\maketitle

\section{Introduction}
\label{section-intro}
\emph{Singular de-Sitter tori} form a
new class of geometric structures
which constitute the Lorentzian analog
of hyperbolic surfaces.
They are modelled on the
two-dimensional \emph{de-Sitter space} $\dS$,
which is the
\emph{space of oriented geodesics
of the hyperbolic plane $\Hn{2}$}
endowed with its natural geometry,
which happens to be \emph{Lorentzian}
and of constant non-zero curvature
(a projective model of $\dS$ is described in
Paragraph \ref{subsection-deSitterspace}).
De-Sitter geometry is in this sense
dual to hyperbolic geometry:
it is \emph{the geometry of hyperbolic geodesics}.
The present paper contributes to the description
of the \emph{deformation space of singular de-Sitter tori}
through the lengths of their geodesics.
\par A natural isometry invariant of a closed hyperbolic surface
is given by the length of the unique geodesic representative
in a given free homotopy class.
This forms the \emph{marked length spectrum}
of the hyperbolic structure which is famously \emph{rigid}.
More precisely, the lengths of a \emph{finite} number
of geodesics entirely determine
a hyperbolic structure
up to isotopy \cite{kleinVorlesungenUeberTheorie1896}.
Unlike the Riemannian case,
existence and uniqueness of
geodesics of non-zero length
does not hold in any free homotopy class of a singular de-Sitter torus.
We will however see
that \emph{closed timelike geodesics}
exist and are unique in any eligible free homotopy class
of a de-Sitter torus with a unique singularity
(Proposition \ref{theoremintro-existenceuniquenessgeodesics}).
This defines the \emph{timelike marked length spectrum}
of such a torus.
In this paper, we
show that
\emph{the lengths of a pair of
simple closed timelike geodesics
of intersection number one
determine
up to isotopy
a de-Sitter structure with a unique singularity
of fixed angle}
(Theorem \ref{theoremintro-rigidityMLS}).

\subsection{De-Sitter tori and their
lightlike foliations}
\label{subsection-introsingulardeSittertorilightlike}
An important ingredient of de-Sitter geometry
is the pair of transverse one-dimensional foliations
of $\dS$
defined by the traces of (weak) stable and unstable
horocycles in the space of geodesics of $\Hn{2}$.
These foliations are called lightlike,
and they define on any surface locally modelled on $\dS$
a \emph{lightlike bi-foliation} $(\Falpha,\Fbeta)$.
The existence of these foliations imposes a strong
topological restriction on a closed
(and orientable)
Lorentzian surface: it has to be homeomorphic to the two-torus
$\Tn{2}=\R^2/\Z^2$,
due to the Poincaré-Hopf theorem.
A Lorentzian version of the Gau{\ss}-Bonnet formula moreover
imposes an additional constraint linking the topology
and the curvature of
a Lorentzian surface,
which forbids the existence of a
regular de-Sitter metric on the torus.
To
investigate
closed Lorentzian surfaces of non-zero constant curvature,
it is thus necessary to introduce
\emph{singularities}.
The natural local notion of
singularities
was described in previous works
\cite[p.160]{barbot_collisions_2011},
inspired by
the classical definition of a
conical Riemannian singularity.
The first examples of
\emph{singular de-Sitter tori}
were given thereafter in
\cite{mion-moutonRigiditySingularDeSitter2024},
where the global study of
singular Lorentzian surfaces was
initiated through the lens
of the lightlike foliations dynamics.
In addition to their Lorentzian nature
and their ``negative-curvature flavour'',
some aspects
of singular de-Sitter tori are reminiscent of
dilation surfaces \cite{boulanger_closed_2025}.
The first-return maps of their lightlike foliations
are however not affine but \emph{homographic
interval exchange transformations}
\cite[\S 4]{mion-moutonRigiditySingularDeSitter2024}.
\par The lightlike bi-foliation
of a de-Sitter torus can be interpreted
as the \emph{conformal class} of
the underlying Lorentzian metric.
The classical \emph{uniformization} question
translates then as follows in the Lorentzian setting:
\emph{is any bi-foliation of the torus
the lightlike bi-foliation
of a unique singular de-Sitter structure?}
This question was answered positively
in \cite{mion-moutonRigiditySingularDeSitter2024}
for a unique singularity and
in the case of \emph{minimal} foliations.
This gave a
partial description of the \emph{deformation space}
$\Deftheta$ of de-Sitter structures of area $\theta>0$
and with a unique singularity
on the torus $\Tn{2}$.
However, the dynamics of the lightlike
foliations does not capture
the full geometry of singular de-Sitter
tori.
There exists indeed non-empty open subsets
of $\Deftheta$ in which the lightlike foliations
are pairwise isotopic
(and have closed leaves).
One of the goals of the present work is to
provide a description
of singular de-Sitter tori
complementary to that of
\cite{mion-moutonRigiditySingularDeSitter2024},
focusing no longer on lightlike foliations
but on \emph{timelike geodesics}.
In particular, we will distinguish
singular de-Sitter structures
having the same lightlike dynamics.

\subsection{Uniqueness of closed geodesics
in singular de-Sitter tori}
\label{subsection-introclosedtimelikegeodesics}
Important objects associated with
geometric structures on surfaces
are their \emph{geodesics},
the first natural question being that of existence:
\emph{does every free homotopy class contain a closed geodesic ?}
In a singular de-Sitter torus,
geodesics have different \emph{signatures}
according to the sign of the metric:
they are said \emph{timelike}
if the metric
is negative
on the direction of the geodesic.
The existence question therefore needs to be
specialized to characterize homotopy classes
containing geodesics of a given signature.
This question is answered
in \cite[Appendix A]{mion-moutonRigiditySingularDeSitter2024}:
in a singular de-Sitter torus,
the free homotopy classes
admitting a timelike geodesic representative
are entirely characterized by the
lightlike foliations dynamics.
More precisely,
the lifts to $\R^2$
of the lightlike foliations
of a singular de-Sitter structure $\m$ on $\Tn{2}$
are asymptotic to two oriented lines
called the \emph{projective asymptotic cycles},
which delimit a timelike cone
$\mathcal{C}^\m\subset\R^2$.
A free homotopy class
$a\in\piun{\Tn{2}}\subset\R^2$
contains a closed timelike geodesic
if and only if $a\in\mathcal{C}^\m$,
in which case $a$ is called a \emph{timelike homotopy class}
(see Paragraph \ref{subsubsection-timelikehomotopyclassesexistence}
for more details).
\par This existence result
raises a second natural question inspired by the hyperbolic case:
\emph{is any closed timelike geodesic unique in its free
homotopy class ?}
In the present paper,
we answer this question
positively
for de-Sitter tori
having a unique singularity
and distinct lightlike asymptotic cycles,
which are called \emph{class A}.
\begin{propositionintro}
\label{theoremintro-existenceuniquenessgeodesics}
Any timelike free homotopy class
of a class A de-Sitter torus with a single singularity
contains a unique geodesic.
Moreover, the latter is a
multiple of a simple closed geodesic,
and it maximizes the length among
causal curves within its free homotopy class.
\end{propositionintro}
This result points out a first
``negative-curvature behaviour''
of singular de-Sitter tori.

\subsection{Timelike marked length spectrum
of singular de-Sitter tori
and its rigidity}
\label{subsection-rigidityTMLS}
Proposition \ref{theoremintro-existenceuniquenessgeodesics}
allows the definition of the \emph{timelike marked length spectrum}
of a de-Sitter structure $\m$ with a unique singularity
on $\Tn{2}$,
as the map
$\mathcal{L}^\m$ sending a
timelike homotopy class
$\azero\in\mathcal{C}^\m\cap\piun{\Tn{2}}$ to the length
\begin{equation*}
 \label{equation-TMLS}
 \mathcal{L}^\m(\azero)\coloneqq L^\m(\gamma_\azero^\m)
\end{equation*}
of the unique
closed timelike geodesic $\gamma_\azero^\m$
in the free homotopy class $\azero$.
By construction,
$\mathcal{L}^\m$ only depends on the isotopy class
$\mzero$ of the de-Sitter structure $\m$,
and gives thus rise to a map
\begin{equation*}
 \label{equation-TLMSsurDeftheta}
 \mzero\in\Deftheta\mapsto\mathcal{L}^\mzero.
\end{equation*}
\par This is the natural analog of the
classical \emph{marked length spectrum} of Riemannian manifolds
of negative curvature.
It is natural to ask wether these lengths
entirely determine the metric up to isotopy,
in which case the marked length spectrum is said \emph{rigid}.
For negative curvature Riemannian manifolds,
the question of the
rigidity of the marked length spectrum was popularized
by Burns-Katok in the form of a
conjecture \cite{burnsManifoldsNonpositiveCurvature1985},
which motivated numerous works
providing partial answers
\cite{otalSpectreMarqueLongueurs1990,
crokeRigiditySurfacesNonpositive1990,
hersonskyRigidityDiscreteIsometry1997,
bankovicMarkedlengthspectralRigidityFlat2018,
guillarmouMarkedLengthSpectrum2019}.
The case of hyperbolic surfaces
is however completely distinct from the one
of surfaces with variable negative curvature.
A \emph{finite} number of lengths is indeed sufficient
to describe a hyperbolic metric
\cite{kleinVorlesungenUeberTheorie1896},
while an \emph{infinite} number of them is necessary
with variable curvature.
This simply reflects the fact that the deformation space
of hyperbolic metrics is \emph{finite dimensional},
contrarily to the one of variable negative curvature metrics.
\par In the de-Sitter case,
the domain $\mathcal{C}^\mzero\cap\piun{\Tn{2}}$
of the marked length spectrum
$\mathcal{L}^\mzero$ depends on the structure $\mzero$,
and is exactly equivalent to the
asymptotic cycles
of the lightlike foliations.
According to the uniqueness results
of \cite{mion-moutonRigiditySingularDeSitter2024},
the domain of $\mathcal{L}^\mzero$ alone
already characterizes
$\mzero\in\Deftheta$ when both lightlike foliations
are minimal.
To complement this description,
one could hope to characterize
a de-Sitter structure with only
a part of its timelike marked length spectrum.
The following result
shows that two timelike lengths
are enough to entirely determine
a class A de-Sitter structure with a unique singularity.
\begin{theoremintro}
\label{theoremintro-rigidityMLS}
 Let $\m_1$ and $\m_2$ be two class A
 de-Sitter structures
 on the torus
 of equal areas and
 having a unique singularity.
 If $\m_1$ and $\m_2$
 give the same lengths to a common pair
 of timelike free homotopy classes
 of intersection number one,
 then they are isotopic.
\end{theoremintro}

\subsection{Length-twist coordinates
on the deformation space of singular de-Sitter tori}
\label{subsection-introfenchennielsen}
As with any deformation space of geometric structures,
one may want to have
natural coordinates
on $\Deftheta$
expressing geometrical parameters of the de-Sitter structures.
On the classical Teichmüller space of hyperbolic structures
on a surface,
this is for instance furnished
by the well-known \emph{Fenchel-Nielsen coordinates}.
In the de-Sitter case that we are interested in,
the Ehresman-Thurston principle shows that
$\Deftheta$ is locally homeomorphic to
a two-dimensional character variety.
We are thus
looking for two natural geometric quantities describing $\Deftheta$,
which explains why \emph{two lengths}
are sufficient to describe a singular de-Sitter
structure in
Theorem \ref{theoremintro-rigidityMLS}.
In contrast to the Riemannian case,
the usual specificity of the Lorentzian signature
prevents
to describe the full deformation space
by length-twist coordinates.
We fix a primitive homotopy
class $\azero\in\piun{\Tn{2}}$,
and restrict to the open subset $\Deftheta^A_{\azero}$
formed by class $A$
de-Sitter structures
for which $\azero$ is a timelike class.
The length function
\begin{equation*}
 \label{equation-lengthfunctionsectionfenchennielsen}
 \mathcal{L}_\azero\colon\mzero\in\Deftheta^A_{\azero}\mapsto
 \mathcal{L}^\mzero(\azero)\in\R_+^*
\end{equation*}
of $\azero$ is well defined
on $\Deftheta^A_{\azero}$,
and gives us our first coordinate.
Observe that Theorem \ref{theoremintro-rigidityMLS}
can be rephrased by saying that
$\mathcal{L}_\azero\times\mathcal{L}_\bzero$
is injective on
$\Deftheta^A_{\azero}\cap \Deftheta^A_{\bzero}$.
\par Choosing a second primitive class $\bzero$ to obtain a
basis $(\azero,\bzero)$ of $\piun{\Tn{2}}$,
we can give a meaning to a second coordinate
$\Theta_{\azero,\bzero}\colon
\Deftheta^A_{\azero}\to\R$
measuring the \emph{twist} around the geodesic in the homotopy class
$\azero$ with respect to the longitude defined by $\bzero$.
The following result
shows that the two previously defined geometrical quantities
provide global coordinates on $\Deftheta^A_{\azero}$.
\begin{theoremintro}
 \label{theoremeintro-fenchennielsen}
 For any basis $(\azero,\bzero)$ of $\piun{\Tn{2}}$,
 the map
 \begin{equation*}
  \label{equation-fenchnielsenintro}
  \mathcal{L}_\azero\times\Theta_{\azero,\bzero}
  \colon\Deftheta^A_{\azero}\to\R_+^*\times\R
 \end{equation*}
is a global homeomorphism.
\end{theoremintro}
These length-twist coordinates give a natural
real-analytic atlas on $\Deftheta^A$
(see Corollary \ref{corollary-DefthetaAhausdorfftopologicalsurface}
and Remark \ref{remark-coordonneesdiff}
for more details).

\subsection{Methods and perspectives}
We highlight in this text some
fundamental geometrical
properties of two-dimensional de-Sitter geometry,
which do not appear to have been previously exploited
and will play an important role in this paper.
We also introduce new elementary tools
for the study of singular de-Sitter tori.
Our hope is to illustrate
that while de-Sitter geometry is fundamentally
distinct from hyperbolic geometry
(essentially
due to the existence of invariant foliations),
many efficient methods are however available
which conceptually replace the usual ``toolbox''
of a hyperbolic geometer.
We wish thereby to motivate
the further investigation of
these structures.
\par The study of singular de-Sitter surfaces
remains in its infancy,
and it is our opinion that the present work
has opened many questions.
The first broad direction concerns
singular de-Sitter structures on
\emph{surfaces of genus higher than two}.
We refer for instance to
\cite[Remark 4.5]{mion-moutonRigiditySingularDeSitter2024}
for a discussion of
the new singularities to be introduced on
such surfaces
(in order to introduce singularities of the lightlike foliations).
In their case even the existence
of closed timelike geodesics
is generally not ensured,
which is the subject of an ongoing work.
The coincidence of de-Sitter geometry and
non-trivial topology
is likely to make the study
of the lengths of closed timelike geodesics
particularly interesting,
concerning the rigidity question studied
in the present paper
as well as the \emph{asymptotic counting}
of such geodesics.
Another promising perspective for these questions
is the consideration of \emph{singular Lorentzian surfaces
with variable curvature of constant sign},
which are the Lorentzian counterpart
of Riemannian surfaces with negative curvature.\footref{footnote-courburedS}
In a future work, we wish
to investigate the links of such structures
with dynamical properties of their
timelike geodesic flow.
\par We focus
in the present paper
on de-Sitter tori with a \emph{unique} singularity.
The main reason
for this is the possibility for
a de-Sitter torus with multiple singularities
to contain distinct freely homotopic
closed timelike geodesics
which ``separate'' the singularities.
Outside of this difficulty,
Lemma \ref{lemma-decoupagetoredeSitter}
holds and yields partial coordinates
on the deformation space,
by arguments similar to Paragraph
\ref{subsection-lengthtwist}.
The investigation of the specific behaviours
appearing in presence of
multiple singularities will be the object
of a forthcoming work.

\subsection{Organization of the paper}
We give in Section \ref{section-deSittertori}
an overview of de-Sitter tori,
sufficient for the purposes of this text.
In Section \ref{section-uniquenessclosedtimelikegeodesics},
we introduce the notions concerning closed timelike geodesics
and prove Proposition \ref{theoremintro-existenceuniquenessgeodesics}.
We define the twist coordinate and prove Theorem
\ref{theoremeintro-fenchennielsen} in
Section \ref{section-FenchenNielsencoordinates}.
We finally prove
Theorem \ref{theoremintro-rigidityMLS}
in Section \ref{section-rigiditytimelikeMLS}.

\subsection*{Acknowledgments}
The author is grateful to Anna Wienhard for originally
drawing his attention to the question of a
``Lorentzian marked length spectrum''.
He also thanks Clément Pérault
for his interest concerning
this work.

\section{Singular de-Sitter tori}
\label{section-deSittertori}
\subsection{De-Sitter space}
\label{subsection-deSitterspace}
We denote by
$\R^{1,2}=(\R^3,q_{1,2})$
the three-dimensional \emph{Minkowski space},
endowed with the Lorentzian quadratic form
$q_{1,2}(x,y,z)=x^2+y^2-z^2$.
A vector $u\in\R^{1,2}$ is respectively called:
\begin{itemize}
\item \emph{spacelike} if $q_{1,2}(u)>0$,
\emph{timelike} if $q_{1,2}(u)<0$,
\item \emph{definite} if $u$ is timelike or spacelike,
\item \emph{lightlike} if $q_{1,2}(u)=0$,
\item \emph{causal} is $q_{1,2}(u)\leq0$,
and \emph{anticausal} is $q_{1,2}(u)\geq0$.
\end{itemize}
Half-lines of $\R^{1,2}$ are named accordingly.
A plane $P\subset\R^{1,2}$ is called:
\begin{itemize}
 \item \emph{timelike} if ${q_{1,2}}\restreinta_P$
 has \emph{Lorentzian} signature $(-,+)$,
 \emph{i.e.} contains two lightlike lines;
 \item \emph{lightlike} if $q_{1,2}\restreinta_P$
 is degenerated,
 \emph{i.e.} contains a unique lightlike line;
 \item \emph{spacelike} if $q_{1,2}\restreinta_P$
 has \emph{Euclidean} signature $(+,+)$,
 \emph{i.e.} contains no lightlike line.
\end{itemize}
The timelike cone of $\R^{1,2}$
has two connected
components, and vectors in the component of
$(0,0,1)$ (respectively $(0,0,-1)$)
are called \emph{future} (resp. \emph{past}) timelike vectors.
The choice of such a future connected component
in the timelike cone is called a \emph{time-orientation}.
\subsubsection{De-Sitter space,
lightlike foliations and their orientations}
\label{subsubsection-deSitterspacefoliations}
We denote by $\mathbf{P}^+(\R^{1,2})$
(homeomorphic to $\Sn{2}$)
the space of half-lines
of $\R^{1,2}$.
The two-dimensional \emph{de Sitter space}
is the quadric
\begin{equation*}
\label{equation-deSitter}
 \dS\coloneqq
 \{d\in\mathbf{P}^+(\R^{1,2})\text{~spacelike}\}
 \subset\mathbf{P}^+(\R^{1,2})
\end{equation*}
endowed with the \emph{Lorentzian metric}
$\mudS$
induced by the restriction to $\dS$
of $q_{1,2}$, which
has constant curvature $1$.
This makes
the de-Sitter space
$\dS$ the Lorentzian analog of
the \emph{hyperboloid model of the
hyperbolic plane}\footnote{\label{footnote-courburedS}Note that
if $g$ is a Lorentzian metric of constant curvature
$1$ on a surface, then $-g$
is also a Lorentzian metric,
and of constant curvature $-1$.
Unlike the Riemannian case,
there are therefore \emph{only two local
constant curvature
Lorentzian geometries in dimension $2$}
(up to anti-isometries and scaling by a constant
factor):
the linear model of curvature zero
(the Minkowski plane $\R^{1,1}$),
and the model of non-zero curvature
(the de-Sitter space $\dS$).}:
\begin{equation*}
 \label{equation-H2}
 \Hn{2}\coloneqq
 \{d\in\mathbf{P}^+(\R^{1,2})\text{~timelike future}\}
 \subset\mathbf{P}^+(\R^{1,2}).
\end{equation*}
Each timelike plane of $\R^{1,2}$
inherits a time-orientation
(whose future cone is contained within the future cone
of $\R^{1,2}$).
We endow $\dS$ with the induced
future timelike and spacelike cones,
and with the orientation coming from
that of $\R^3$
(with the exterior normal rule).
The \emph{lightlike foliations} $\mathcal{F}_{\alpha/\beta}$
of $\dS$ are the one tangent to its two lightlike
line fields.
They are named and oriented
compatibly with the orientation and the time-orientation
of $\dS$
as indicated in Figure \ref{figure-definitionXsingulartheta}.
\subsubsection{Geodesics in the projective de-Sitter model}
\label{subsubsection-geodesicsprojectivedS}
The model $\dS$ of the de-Sitter space is
\emph{projective}
in the sense that the geodesics
of $\dS$ are precisely the connected components
of the intersections of projective lines
of $\mathbf{P}^+(\R^{1,2})$ with $\dS$.
Let
\begin{equation*}
\label{equation-bordinfinidS}
 \partial_\infty^\pm\dS\coloneqq
 \{d\in\mathbf{P}^+(\R^{1,2})\text{~lightlike
 future/past}\}
 \subset\mathbf{P}^+(\R^{1,2})
\end{equation*}
denote the
\emph{future} (resp. past) \emph{boundary
of $\dS$ at infinity},
which we
endow with the
orientation compatible with the one
of $\dS$.
The geodesics of $\dS$
are described as follows.
\begin{enumerate}
\item \emph{Timelike projective lines}
 are the one intersecting each of the two boundaries
 $\partial_\infty^\pm\dS$ in two points.
 Their intersections with $\dS$
 have two connected components which are opposite
 \emph{timelike geodesics}.
 \item \emph{Lightlike projective lines}
  are the lines $L$ tangent to
 $\partial_\infty^\pm\dS$,
 \emph{i.e.} intersecting each of the two boundaries
 at a unique point $L^\pm$.
 Their intersections with $\dS$
 have two connected components.
 \begin{itemize}
  \item The component
  converging to $\partial_\infty^-\dS$
  in the future is an \emph{$\alpha$ lightlike geodesic}.
  \item The component
  converging to $\partial_\infty^+\dS$
  in the future is a \emph{$\beta$ lightlike geodesic}.
 \end{itemize}
 \item \emph{Spacelike projective lines}
 are the one disjoint from
 $\partial_\infty^\pm\dS$, and coincide with
 \emph{spacelike geodesics} of $\dS$.
\end{enumerate}

The complement of any spacelike geodesic $s$ of $\dS$
in $\mathbf{P}^+(\R^{1,2})$
is the union of a future affine chart
$\mathcal{A}$
and of a past affine chart,
in which $\partial_\infty^\pm\dS$
is a round circle, and
of which $\Hn{2}$ (respectively its past copy)
is the interior
and $\dS_s\coloneqq \dS\cap\mathcal{A}$ the exterior.
The traces in $\mathcal{A}$ of
timelike, lightlike and spacelike geodesics
of $\dS$ are the intervals in $\dS_s$
of the affine lines of $\mathcal{A}_s$
which respectively
intersect,
are tangent, and are disjoint
from the future boundary $\partial_\infty^+\dS$.
We refer to
\cite[Figure 5 p.13]{nurowskiBeltramiSitterModel2025}
for a representation of such an affine chart,
and for a thorough comparison
of different models of $\dS$.

\subsubsection{De-Sitter plane and isometries}
\label{subsubsection-deSitterplane}
The universal
cover of the de-Sitter space,
homeomorphic to $\R^2$,
is denoted by
$\dStilde$ and called the \emph{de-Sitter plane}.
We endow $\dStilde$ with the pullback of
the metric of $\dS$.
The automorphism group of the
universal cover
$\dStilde\to\dS$
is isomorphic to $\Z$ and generated
by the action of a (closed) spacelike geodesic of $\dS$
on the universal cover.
\par The connected component of the identity
in the stabiliser of $q_{1,2}$
in $\SL{3}$ is denoted by
$\Bisozero{1}{2}$.
It is the group of isometries
of $\dS$
(and of $\Hn{2}$)
preserving both its orientation and its
time-orientation
(see for instance
\cite[Lemma 2.2]{mion-moutonRigiditySingularDeSitter2024}).

\subsubsection{Identification with the space
of oriented geodesics of $\Hn{2}$}
\label{subsubsection-spaceorientedgeodesics}
Any $p\in\dS$ is contained in
a unique $\alpha$ (respectively $\beta$) lightlike
geodesic $p_\alpha$ (resp. $p_\beta$)
of $\dS$,
having a past limit point
$p_\alpha^\infty\in\partial_\infty^+\dS$
(resp. a future limit point
$p_\beta^\infty\in\partial_\infty^+\dS$).
This defines two
$\Bisozero{1}{2}$-equivariant
projections
\begin{equation*}
\label{equation-projectionsCalphabeta}
\pi_{\alpha/\beta}\colon
p\in\dS\mapsto
p_{\alpha/\beta}^\infty\in\partial_\infty^+\dS
\end{equation*}
whose fibers are precisely the $\alpha$
(respectively $\beta$)
lightlike geodesics of $\dS$.
Note that $p_\alpha^\infty\neq p_\beta^\infty$,
namely
\begin{equation*}
\label{equation-espacegeodesiquesH2}
 (\pi_\alpha\times\pi_\beta)(p)\in(\partial_\infty^+\dS)^{(2)}
\coloneqq(\partial_\infty^+\dS)^2
\setminus\enstq{(d,d)}{d\in\partial_\infty^+\dS}.
\end{equation*}
Moreover
the $\Bisozero{1}{2}$-equivariant map
\begin{equation}
\label{equation-identificationespacegeodesiquesH2}
 \pi_\alpha\times\pi_\beta
 \colon\dS\to
 (\partial_\infty^+\dS)^{(2)}
\end{equation}
is a bijection.
Identifying $\partial_\infty^+\dS$ with the boundary
of $\Hn{2}$,
$(\partial_\infty^+\dS)^{(2)}$ is
the space of oriented geodesics of $\Hn{2}$
which is therefore identified with $\dS$
in a $\Bisozero{1}{2}$-equivariant way.
The geodesic
of $\Hn{2}$
containing $(\pi_\alpha\times\pi_\beta)(p)$
and oriented from
$p_\alpha^\infty$ to $p_\beta^\infty$
is simply the intersection with $\Hn{2}$
of the orthogonal of $p$ in $\R^{1,2}$
(oriented compatibly with $p$
and with the orientation of $\R^{1,2}$).
\begin{remark}
 \label{remark-dSuniquenaturalgeometry}
 One can check that $\mudS$ is the unique
pseudo-Riemannian metric of $\dS$
(up to scaling by a constant factor)
which is invariant by $\Bisozero{1}{2}$
(see
\cite[\S 2.3]{mion-moutonRigiditySingularDeSitter2024}
for instance).
The de-Sitter geometry
studied in this article
is therefore the unique isometry-invariant
geometry of the space of oriented geodesics
of the hyperbolic plane.
\end{remark}

\subsection{Singular de-Sitter tori}
\label{subsection-singulardeSittertori}
We introduce in this paragraph
all the notions about
singular de-Sitter tori
needed in this text.
We refer to
\cite[\S 3 and \S 4]{mion-moutonRigiditySingularDeSitter2024}
for more details and
for proofs of the claims appearing in this
paragraph.

\subsubsection{De-Sitter surfaces and holonomy morphism}
\label{subsubsection-deSittersurfaces}
A \emph{de-Sitter surface} is a surface locally
modelled on $\dS$, in a sense
made rigorous by the following notion
of \emph{$(G,X)$-structure}.
\begin{definition}\label{equation-definitionsurfaceslocalementdS}
A \emph{$\dS$-atlas} on an oriented
topological surface $S$ is an
atlas of orientation-preserving
topological charts $\varphi_i\colon U_i\to\dS$
from connected open subsets $U_i\subset S$
to $\dS$
(called \emph{$\dS$-charts}),
whose
transition maps
$\varphi_j\circ\varphi_i^{-1}\colon \varphi_j(U_i\cap U_j)\to
\varphi_i(U_i\cap U_j)$
equal
the restriction of an element of $\Bisozero{1}{2}$.
A \emph{$\dS$-structure} on $S$
is a maximal $\dS$-atlas,
and a \emph{$\dS$-surface} is an oriented surface
endowed with a $\dS$-structure.
\end{definition}
As in the case of hyperbolic surfaces,
one can show that the data of a $\dS$-structure
on $S$ is exactly equivalent
to that of a time-oriented
Lorentzian metric of constant curvature $1$
(see \cite[Proposition-Definition 2.5]{mion-moutonRigiditySingularDeSitter2024}).
\par The restriction of the
universal covering map
to sufficiently small open subsets
endows the universal cover $\tilde{S}$
of a $\dS$-surface $S$
with an induced $\dS$-structure.
For instance, $\dStilde$ has a natural $\dS$-structure.
The fundamental group $\piun{S}$ of $S$
acts on $\tilde{S}$ by
isometries, and this action
is encoded through
the \emph{holonomy morphism}
\begin{equation*}
 \label{equation-holonomy}
 \hol\colon\piun{S}\to\Bisozero{1}{2}
\end{equation*}
of $S$.
The latter is characterized by the existence of a
\emph{developing map}
$\delta\colon\tilde{S}\to\dS$
whose restriction to any sufficiently small subset
is a $\dS$-chart, and which
is equivariant with respect to the holonomy morphism
$\hol$.

\subsubsection{Singular de-Sitter surfaces and their lightlike foliations}
\label{subsubsection-localsingularities}
We now define the local model of singularities,
originally introduced in \cite[p.160]{barbot_collisions_2011}
and inspired from classical Riemannian conical
singularities.
Let $\odS\in\dStilde$ be a fixed base-point
whose stabilizer in $\Bisozero{1}{2}$
is denoted by $\{a^u\}_{u\in\R}$.
We fix $\theta>0$.
Any future timelike half-geodesic
$\gamma\subset\dStilde$
emanating from $\odS$
bounds together with
$a^\theta(\gamma)$
an open sector $D_\theta$
of $\dStilde$ of angle $\theta$
indicated in Figure
\ref{figure-definitionXsingulartheta}.
The complement
$\dStilde\setminus D_\theta$
is a $\dS$-surface with two timelike geodesic
boundary components
$\gamma$ and $a^\theta(\gamma)$
and a conical point $\odS$.
The quotient
\begin{equation*}
 \label{equation-dSsingularthetatilde}
 \dSsingularthetatilde\coloneqq
 (\dStilde\setminus D_\theta)/\sim_\theta
\end{equation*}
by the identification
$\gamma\ni x\sim_\theta
a^\theta(x)\in a^\theta(\gamma)$
of the boundary components
is the \emph{standard $\dS$-cone}
of angle $\theta$.
It bears a marked point
$\odSsingulartheta\in\dSsingularthetatilde$
which is the projection of $\odS$.
The identification
being made by isometries,
$\dSsingularthetatilde\setminus\{\odSsingulartheta\}$
has a natural $\dS$-structure
induced by that of $\dStilde$.
Note that $\dSsingularthetatilde$
remains homeomorphic to $\R^2$.
A similar construction defines
the local singularity for $-\theta$,
but we will only need positive singularities
in this text as we will see
in Proposition \ref{proposition-GaussBonnet}.
The local singularity can be described in various other ways
as detailed in
\cite[\S 3.1]{mion-moutonRigiditySingularDeSitter2024}.
\begin{figure}[!h]
	\begin{center}
		\def\svgwidth{0.65 \columnwidth}
			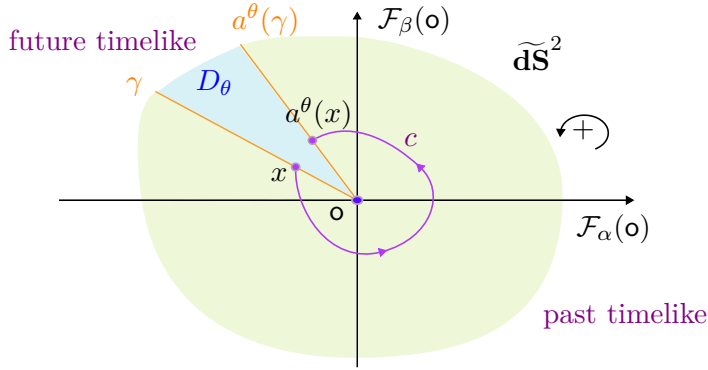
	\end{center}
	\caption{Local de-Sitter singularity
	of angle $\theta$.}
			\label{figure-definitionXsingulartheta}
\end{figure}
\par By the very definition of $\dSsingularthetatilde$,
the holonomy
of a small curve $c$
around $\odSsingulartheta$
equals $a^\theta$
(see Figure \ref{figure-definitionXsingulartheta},
and \cite[\S 3.1.3]{mion-moutonRigiditySingularDeSitter2024}
for more details).
Since $a^\theta\neq\id$,
this shows that
the $\dS$-structure of
$\dSsingularthetatilde\setminus\{\odSsingulartheta\}$
cannot be extended to
$\odSsingulartheta$.
Indeed if it could,
then the curve $c$ would be homotopically trivial
in the $\dS$-surface $\dSsingularthetatilde$,
and its holonomy would thus be trivial as
well.
While this argument may
be unclear
to a reader not
used to the notion of
$(G,X)$-structures,
it partially explains
the importance of the holonomy
for singular $\dS$-structures
that we now define.
\begin{definition}[{\cite[Definition 3.16]{mion-moutonRigiditySingularDeSitter2024}}]
 \label{definition-singulardSstructures}
 A \emph{singular $\dS$-structure} $\m$
 with \emph{singular points} $\Sigma\subset S$
 on an oriented surface $S$
 is a $\dS$-structure on $S\setminus\Sigma$
 whose singular points are
 \emph{standard singularities}
 in the following sense.
 For any $x\in\Sigma$,
 there exists
 $\theta_x\in\R^*$ (the \emph{angle at $x$})
 and a homeomorphism
 $\varphi$
 (called a \emph{singular $\dS$-chart}
 at $x$)
 from an open neighbourhood
 $U\subset S$ of $x$
 to an open neighbourhood
 $V\subset\dSsingularthetatilde$ of $\odSsingulartheta$
 such that:
 \begin{enumerate}
  \item $U\cap\Sigma=\{x\}$,
  \item $\varphi(x)=\odSsingulartheta$,
  \item and $\varphi$ is an isometry in restriction
  to $S\setminus\Sigma$.
 \end{enumerate}
Points $x\in S\setminus\Sigma$ are said
\emph{regular}.
A \emph{singular $\dS$-structure with
geodesic boundary} on an oriented surface
$S$ with boundary
is a singular $\dS$-structure
on the interior of $S$,
of which the boundary is geodesic
(for the induced Lorentzian metric).
\end{definition}
The following statement,
proved in
\cite[Lemma 3.5]{mion-moutonRigiditySingularDeSitter2024},
gives a useful way to
know if a potential singular point
is actually regular or not.
\begin{lemma}[{\cite[Lemma 3.5]{mion-moutonRigiditySingularDeSitter2024}}]
 \label{lemma-dSstructuresholonomy}
Let $S$ be a singular $\dS$-surface
of holonomy morphism $\hol$.
Let $x\in S$ be a standard singularity
of angle $\theta$,
and $c_x$ be the homotopy class
of a small closed curve around $x$.
Then $\hol(c_x)$ is conjugated to $a^\theta$.
In particular, $x$ is a regular point
if and only if $\hol(c_x)=\id$.
\end{lemma}
An important property
of a singular $\dS$-structure $\m$
on a surface $S$
is that the lightlike foliations
induced by $\m$
outside of the singularities
extend on $S$ to two transverse
\emph{topological} one-dimensional foliations
(see \cite[Lemma 3.24]{mion-moutonRigiditySingularDeSitter2024}).
These foliations are still called the
\emph{lightlike foliations of $S$},
and are denoted by $(\Falpha^\m,\Fbeta^\m)$.

\subsubsection{Lorentzian angles and Gau{\ss}-Bonnet formula}
\label{subsubsection-Lorentzianangles}
We will use in the sequel of the text
a notion of Lorentzian angles
that we now introduce
following
\cite{birmanGaussBonnetTheorem2dimensional1984}.
\begin{definition}[\cite{birmanGaussBonnetTheorem2dimensional1984}]\label{definition-timespaceLorentzianangle}
Let $P$ be an oriented plane endowed with a Lorentzian scalar product
$\prodscal{\cdot}{\cdot}$.
For $X,Y\in P$, we denote
$\orientation{X}{Y}=1$ (respectively $-1$)
if $(X,Y)$ is a positively (resp. negatively) oriented basis,
and $\orientation{X}{Y}=0$ if $(X,Y)$ are linearly dependent.
Then for $(X,Y)$ two unit timelike vectors
belonging to the same quadrant
of $P$,
the Lorentzian \emph{angle from $X$ to $Y$} is defined by
\begin{equation*}
\label{equation-definitionanglememequadrant}
\angleLorentzien{X}{Y}\coloneqq \orientation{X}{Y}
\arcosh\abs{\prodscal{X}{Y}}
\end{equation*}
with $\arcosh\colon\intervallefo{1}{+\infty}\to\R^+$
the inverse hyperbolic cosine function.
This definition is extended to any pair $(X,Y)$
of unit timelike vectors by the relation
\begin{equation*}\label{equation-definitionangledifferentsquadrant}
\angleLorentzien{X}{Y}=\angleLorentzien{X}{-Y}.
\end{equation*}
For $a$ and $b$ any two timelike
(half-)geodesics of $\dS$
emanating from a point $p\in\dS$,
we will denote by
$\angleLorentzien{a}{b}_p
=\angleLorentzien{\Tan{p}{a}}{\Tan{p}{b}}$
the angle from $a$ to $b$ at $p$.
\end{definition}
Note that \eqref{equation-definitionanglememequadrant} is well-defined
since $\abs{\prodscal{X}{Y}}\geq1$ according to the Lorentzian
Cauchy-Schwartz inequality.
Furthermore for any two unit timelike vectors $X,Y$, the relations
\begin{equation}\label{equation-relationsangle}
 \angleLorentzien{-X}{-Y}=\angleLorentzien{-X}{Y}
 =-\angleLorentzien{Y}{X}
 =\angleLorentzien{X}{Y}
\end{equation}
follow easily from the definition
(see \cite[Lemma 1]{birmanGaussBonnetTheorem2dimensional1984}).
This notion of angles allows to prove
the following Gau{\ss}-Bonnet
formula for simply connected domains,
due to Birman-Nomizu \cite{birmanGaussBonnetTheorem2dimensional1984}.
\begin{proposition}[{\cite[p.80]{birmanGaussBonnetTheorem2dimensional1984}}]
 \label{proposition-GaussBonnetBirmanNomizu}
 Let $P\subset\dStilde$ be a
 compact subset whose boundary is
 a polygon with timelike geodesic
 edges $(E_1,\dots,E_n)$ ($n\geq 3$)
 of respective endpoints $(v_1,\dots,v_n)$.
 Let $\nu_i=\angleLorentzien{E_i}{E_{i+1}}_{v_i}$
 denote the exterior angle
 at $v_i$
 for $i=1,\dots,n$
 (with $E_{n+1}\coloneqq E_1$).
 The area of $P$ equals
 the sum of the exterior angles:
 \begin{equation*}
  \label{equation-GaussBonnetBirmanNomizu}
  \mathcal{A}(P)=\sum_{i=1}^n\nu_i.
 \end{equation*}
\end{proposition}
Using this relation,
the following Gau{\ss}-Bonnet formula
was obtained
for singular $\dS$-surfaces
with geodesic boundary
in \cite[Proposition 3.28]{mion-moutonRigiditySingularDeSitter2024}.
\begin{proposition}[{\cite[Proposition 3.28]{mion-moutonRigiditySingularDeSitter2024}}]
 \label{proposition-GaussBonnet}
 Let $S$ be a compact
 singular $\dS$-surface
 with timelike geodesic boundary
 and $n$ singularities of angles
 $(\theta_1,\dots,\theta_n)$.
 The area of $S$ equals
 \begin{equation*}
  \label{equation-Gaussbonnet}
  \mathcal{A}(S)=\sum_{i=1}^n\theta_i.
 \end{equation*}
If $S$ has a unique singularity $x$,
$x$ has thus a positive angle equal to the
area of $S$.
\end{proposition}

\subsubsection{Examples of singular de-Sitter tori}
We conclude this section
with examples of
$\dS$-tori with a unique singularity.
Recall first that $\dS$ identifies
$\Bisozero{1}{2}$-equivariantly
with $(\partial_\infty^+\dS)^{(2)}$
according to
\eqref{equation-identificationespacegeodesiquesH2}.
Since $\partial_\infty^+\dS$
can itself be projectively identified with
$\RP{1}$,
$\dS$ is eventually identified
with
\begin{equation*}
 \dSancien\coloneqq(\RP{1})^{(2)},
\end{equation*}
equivariantly with respect to an isomorphism
from $\Bisozero{1}{2}$ to $\PSL{2}$
\cite[Remark 2.2]{mion-moutonRigiditySingularDeSitter2024}.
We work in this paragraph with the model
$\dSancien$
that we endow with its unique
$\PSL{2}$-invariant
Lorentzian metric of constant curvature $1$
\cite[\S 2.3]{mion-moutonRigiditySingularDeSitter2024}.
We also identify
$\RP{1}$ to $\R\cup\{\infty\}$,
on which $\PSL{2}$ acts by homographies.
\par For any fixed $\theta>0$,
$x\in\intervalleoo{1}{\infty}$
and $y\in\intervalleoo{0}{1}$,
there exists in $\dSancien$ a
unique polygon
$\mathcal{L}_{\theta,x,y}$
with lightlike edges
and area $\theta$
as indicated in
Figure \ref{figure-gluingLshaped}
\cite[\S 4.3]{mion-moutonRigiditySingularDeSitter2024}.
Any
$x'\in\intervalleoo{1}{\infty}$
and $y'\in\intervalleoo{0}{y_+}$
moreover define unique
identifications of the edges
of $\partial\mathcal{L}_{\theta,x,y}$
by isometries in $\PSL{2}$
as indicated in
Figure \ref{figure-gluingLshaped}.\footnote{This
is due to the fact that
$\PSL{2}$ acts simply transitively
on proper $\alpha$ (respectively $\beta$)
lightlike segments of $\dSancien$.}
The quotient of
$\mathcal{L}_{\theta,x,y}$ by these
edges identifications
is a $\dS$-torus $T_{x,y,x',y'}$
of area $\theta$ and with
at most three singularities
(at the projections of the vertices
of $\mathcal{L}_{\theta,x,y}$).
According to \cite[Lemma 4.10]{mion-moutonRigiditySingularDeSitter2024},
there exists
unique parameters $(x'(x,y),y'(x,y))$ so that
the
(projection of the)
purple point of coordinates $(1,0)$
is the unique singularity of
$T_{x,y,x'(x,y),y'(x,y)}$.
The $\dS$-torus
\begin{equation*}
 \mathcal{T}_{\theta,x,y}\coloneqq T_{x,y,x'(x,y),y'(x,y)}
\end{equation*}
having area $\theta$,
its unique singularity has thus angle
$\theta$ according to
Proposition \ref{proposition-GaussBonnet}.
\begin{figure}[h]
	\begin{center}
		\def\svgwidth{0.8 \columnwidth}
			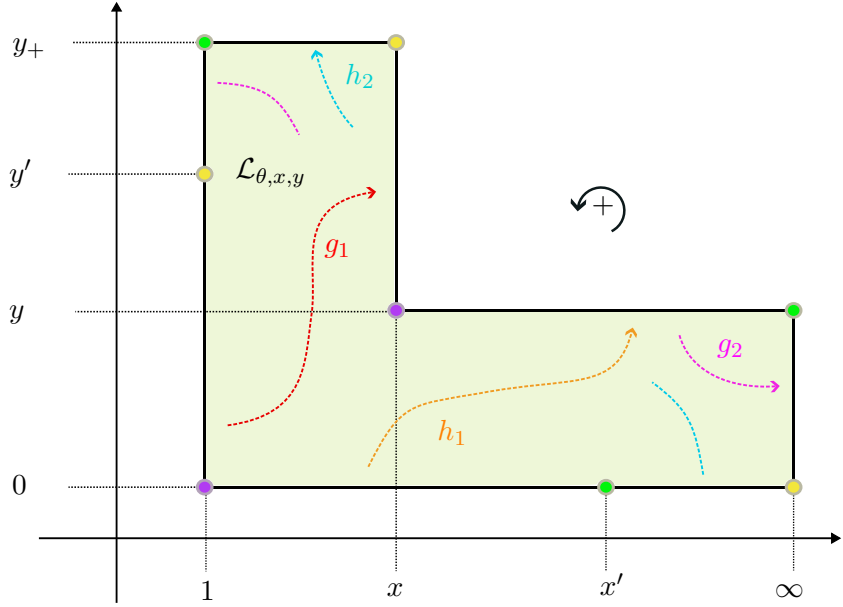
	\end{center}
	\caption{A class A singular $\dS$-torus
	$\mathcal{T}_{\theta,x,y}$.}
			\label{figure-gluingLshaped}
\end{figure}
\par We actually described all the de-Sitter structures
that will be studied in this text,
which are the \emph{class A} structures to be
defined below.\footnote{Note
that it is not known yet
if there exists de-Sitter tori
with a unique singularity which are not
class A.}
According to
\cite[Theorem 9.6]{mion-moutonRigiditySingularDeSitter2024},
any class A $\dS$-torus
of area $\theta$
with a unique singularity
is indeed isometric to one of the
tori $\mathcal{T}_{\theta,x,y}$
that we just constructed.

\section{Timelike geodesics
representatives
in free homotopy classes}
\label{section-uniquenessclosedtimelikegeodesics}
\subsection{Closed timelike geodesics}\label{subsection-existenceclosedtimelikegeodesics}
We introduce in this paragraph the primary object of interest
of this article:
\emph{closed geodesics} of singular de-Sitter surfaces.
\begin{definition}
\label{definition-closedgeodesic}
A \emph{geodesic loop}
of a singular de-Sitter surface $(S,\Sigma)$
is the image of
a closed continuous curve $\gamma\colon\intervalleff{0}{1}\to S$
such that $\gamma\restreinta_{\intervalleoo{0}{1}}$ is a
parametrized geodesic
of the regular de-Sitter surface $S^*\coloneqq S\setminus\Sigma$.
A \emph{closed geodesic} is a geodesic loop such that
$\gamma(0)=\gamma(1)\in S^*$ and
$\gamma'(1)=\gamma'(0)$.
A \emph{multiple} of a closed geodesic $\gamma$
is the closed geodesic $n\gamma$ obtained by travelling
$n$ times $\gamma$ for some $n\in\N^*$.
\end{definition}
Beware that
geodesic loops and closed geodesics are unparametrized
by default,
and that a closed geodesic
is $\cc^\infty$ in this paper
(which is sometimes called a $\cc^1$-closed geodesic
in the literature).
We also stress that
\emph{closed geodesics do not contain
any singularity}.

\subsubsection{Asymptotic cycles and
class A structures}
\label{subsubsection-timelikehomotopyclasses}
In contrast to the Riemannian case,
geodesics have different \emph{signatures}
in Lorentzian surfaces.
We would like to translate this
geometric distinction
homologically,
into a
necessary and sufficient condition
for a free homotopy class to contain
a causal curve.
To this end, we first need to
materialize homologically the lightlike foliations.
This will be in this text the role
of the \emph{oriented projective asymptotic cycle}
of a lightlike foliation $\mathcal{F}_{\alpha/\beta}$
of a de-Sitter structure
on the two-torus $\Tn{2}\coloneqq\R^2/\Z^2$,
which is a half-line
\begin{equation*}
 \label{equation-asymptoticcycle}
 A^+(\mathcal{F}_{\alpha/\beta})\in\mathbf{P}^+(\Homologie{1}{\R}{\Tn{2}})
\end{equation*}
in the homology.
Its \emph{non-oriented projective asymptotic cycle} is the line
$A(\mathcal{F}_{\alpha/\beta})\coloneqq
\R A^+(\mathcal{F}_{\alpha/\beta})$.
We will be interested in this article
with singular de-Sitter structures $\m$
whose lightlike foliations have distinct non-oriented
projective asymptotic cycles
$A(\mathcal{F}^\m_{\alpha})\neq A(\mathcal{F}^\m_{\beta})$,
which are called \emph{class A structures}.
\begin{remark}
 \label{remark-projctiveasymptoticcycle}
The lightlike foliations of a class
A structure $\m$ are
suspensions of circle homeomorphisms
(see \cite[Lemma 6.6]{mion-moutonRigiditySingularDeSitter2024}).
Let $S\subset\Tn{2}$ be any simple closed curve transverse
to $\mathcal{F}^\m_{\alpha/\beta}$.
The first-return map $P_{\alpha/\beta}$
of $\mathcal{F}^\m_{\alpha/\beta}$ on $S$
is a homeomorphism of the circle $S$
which describes entirely
$\mathcal{F}^\m_{\alpha/\beta}$
``locally'',
and whose main topological conjugacy invariant
is the \emph{rotation number}
$\rho(P_{\alpha/\beta})\in\Sn{1}$.
Observe however that the image
of $\mathcal{F}^\m_{\alpha/\beta}$
by a Dehn twist around $S$
is not isotopic to $\mathcal{F}^\m_{\alpha/\beta}$,
while keeping the same first-return map.
One therefore refines
the rotation number of $P_{\alpha/\beta}$
to a \emph{global topological isotopy invariant}
of $\mathcal{F}^\m_{\alpha/\beta}$,
detecting such Dehn twists,
which is its oriented projective asymptotic cycle
$A(\mathcal{F}_{\alpha/\beta})$.
We refer to
\cite[\S 5.2]{mion-moutonRigiditySingularDeSitter2024}
for more details.
\end{remark}
Since the asymptotic cycles are isotopy
invariant, the asymptotic
cycle of the isotopy class
$\mzero$ of a singular
de-Sitter structure $\m$ is well-defined.

\subsubsection{Timelike homotopy classes
and existence of closed timelike geodesics}
\label{subsubsection-timelikehomotopyclassesexistence}
Thanks to asymptotic cycles,
we can now interpretate the three
signatures of Lorentzian geometry
for homotopy classes.
\begin{definition}
\label{definition-timelikeconehomology}
 For any
 (isotopy class of)
 singular de-Sitter structure
 $\mzero$ of $\Tn{2}$,
 \begin{equation*}
 \label{equation-Cmu}
  \mathcal{C}^\mzero_+\coloneqq\Int(\Conv(A^+(\Fbeta^\mzero)\cup-A^+(\Falpha^\mzero)))
  \subset\Homologie{1}{\R}{\Tn{2}}
 \end{equation*}
is the \emph{future timelike cone in homology} of $\mzero$,
which is the convex hull of the half-lines
$A^+(\Fbeta^\mzero)$ and $-A^+(\Falpha^\mzero)$.
We also denote by $\mathcal{C}^\mzero_+\coloneqq-\mathcal{C}^\mzero_-$
the past timelike cone, and by
$\mathcal{C}^\mzero\coloneqq\mathcal{C}^\mzero_+\cup\mathcal{C}^\mzero_-$
the \emph{timelike cone in homology}.
We define likewise the \emph{spacelike cone in homology}
$\mathcal{C}^\mzero_{space}$ of $\mzero$.
\end{definition}
\begin{remarks}
\label{remark-timelikecone}
\begin{enumerate}
 \item One may note that $A^+(\Falpha^\mzero)$ and $A^+(\Fbeta^\mzero)$
are the oriented lightlike lines of
a unique (up to positive scalar multiplication)
Lorentzian quadratic form on
$\Homologie{1}{\R}{\Tn{2}}$,
of which $\mathcal{C}^\mzero$ is the timelike cone.
\item If $\mzero$ is a class B structure \emph{i.e.}
$A(\Falpha^\mzero)=A(\Fbeta^\mzero)$,
then either $\mathcal{C}^\mzero$ or $\mathcal{C}^\mzero_{space}$ is empty.
\end{enumerate}
\end{remarks}

The following existence result is proved
in \cite[Appendix A]{mion-moutonRigiditySingularDeSitter2024}.
It is a reformulation of Proposition A.8,
Corollary A.11
and Theorem A.17 of this article.
\begin{theorem}[{\cite[Appendix A]{mion-moutonRigiditySingularDeSitter2024}}]
\label{theorem-existence}
 Let $\m$ be a class A de-Sitter structure of $\Tn{2}$
 with a unique singularity.
 Then a free homotopy class $c$ of closed curves contains
 a closed timelike geodesic
 which maximizes the length among causal curves
 in $c$,
 if and only if
 $c$ belongs to the timelike cone $\mathcal{C}^\m$ in homology.
\end{theorem}

\subsection{Uniqueness of closed timelike geodesics}\label{subsection-uniquenessclosedtimelikegeodesics}
We make the naive observation that for
any base-point $x\in\Tn{2}$,
the natural map sending a homotopy class
$a\in\pi_1(\Tn{2},x)$
relative to $x$ to the free homotopy class
containing $a$ is an isomorphism.\footnote{Because
$\pi_1(\Tn{2},x)\simeq\Z^2$ is abelian.}
We therefore identify through these maps
the group of free homotopy classes with the
fundamental groups of $\Tn{2}$, which
we denote by $\piun{\Tn{2}}\simeq\Z^2$.
A free homotopy class $a\in\piun{\Tn{2}}$
is said \emph{primitive} if it is not a non-trivial multiple,
\emph{i.e.} if it belongs to
$\piun{\Tn{2}}_*\simeq\Z^2_*\coloneqq\Z^2\setminus\N^*\Z^2$.
We recall that a free homotopy class of $\piun{\Tn{2}}$
is primitive if and only if
it contains a \emph{simple} closed curve
(\emph{i.e.} a closed curve without self-intersections).
We will denote by $[\gamma]$ the free homotopy class of a closed
curve $\gamma$.
\par Unlike the case of hyperbolic surfaces,
we need to show that closed geodesics
in primitive homotopy classes are simple
\emph{before} proving the uniqueness.
\begin{lemma}
\label{lemma-simpletimelikegeodesics}
Let $\gamma$ be a definite closed geodesic of a class A
de-Sitter
torus $T$
with a single singularity.
If the free homotopy class of $\gamma$
in $\piun{T}$
is primitive,
then $\gamma$ is simple.
\end{lemma}
\begin{proof}
We can assume that $\gamma$ is timelike to fix ideas,
the arguments being the same in the spacelike case
by inverting the metric.
Assume by contradiction
that the timelike closed geodesic
$\gamma\subset T$ is not simple.
Since $\gamma$ is
freely homotopic to a simple closed curve
by assumption,
it must then contain a non-trivial
embedded $1$-gon or $2$-gon
according to \cite[Theorem 2.7 p.92]{hassIntersectionsCurvesSurfaces1985}.
A $1$-gon in $\gamma$ would be a
homotopically trivial
timelike closed curve,
which is forbidded by
\cite[Corollary A.6]{mion-moutonRigiditySingularDeSitter2024}.
There exists therefore a closed topological disk $D$ in $T$
whose boundary is the union of two distinct
simple timelike geodesic
segments $a$ and $b$, such that
$a\cap b=\partial a=\partial b$.
Assume that $D$ contains
the unique singularity $p$ of $T$,
of positive angle $\theta$.
Since $p$ is the unique singularity of $D$
and $D$ is simply connected,
any singular $\dS$-chart at $p$
extends to an open neighborhood $U$ of $D$, to give
a topological embedding
$\varphi\colon U\to\dSsingularthetatilde$
which is a $\dS$-chart outside of $p$ and
such that $\varphi(p)=\odSsingulartheta$.
This is due to
\cite[Lemma 3.5]{mion-moutonRigiditySingularDeSitter2024}
in the one hand,
and on the other hand
to the fact that any class A $\dS$-torus
is isometric to a torus
$\mathcal{T}_{\theta,x,y}$ as described
in Figure \ref{figure-gluingLshaped},
according to
\cite[Theorem 9.6.(1)]{mion-moutonRigiditySingularDeSitter2024}.
If $D$ does not contain the singularity,
then such a $\dS$-chart $\varphi$
also exist and we denote it in the same way
with $\theta=0$.
The image by $\varphi$ of $\partial D$ gives
two distinct simple timelike geodesic segments
$a_0$ and $b_0$ in $\dSsingularthetatilde\setminus\{\odSsingulartheta\}$,
which meet two times at their extremities.
But one easily observe that two distinct geodesics in
$\dStilde$ or
$\dSsingularthetatilde\setminus\{\odSsingulartheta\}$
meet at most one
(see for instance
\cite[Figure A.1]{mion-moutonRigiditySingularDeSitter2024}).
This contradiction concludes the proof.
\end{proof}

\subsubsection{Asymptotic points of
timelike geodesics
in the de-Sitter space}
\label{subsubsection-geodesicdeSitterpœlane}
Any timelike geodesic $l$ of $\dS$ has a unique
\emph{future limit point $l^+\in\partial_\infty^+\dS$}
(respectively \emph{past limit point
$l^-\in\partial_\infty^-\dS$})
in $\mathbf{P}^+(\R^{1,2})$.
Moreover $l^+\neq -l^-$
and $l$ is uniquely described by
$(l^+,l^-)$:
$(l_1^+,l_1^-)=(l_2^+,l_2^-)$
if and only if $l_1=l_2$.
In other words, the space of timelike geodesics
of $\dS$ identifies with $\dS$
(by taking the orthogonal in $\R^{1,2}$),
hence with the space of oriented geodesics of $\Hn{2}$.
\par Let $L_1$ and $L_2$ be
two timelike planes of $\R^{1,2}$,
and $L_i^\pm$ be the four corresponding timelike geodesics
of $\dS$.
The three possible configurations
of $L_1$ and $L_2$
are the following.
\begin{enumerate}
 \item \emph{$L_1\cap L_2$ is lightlike},
 and contains thus one future and one past
 half-line $(L_1\cap L_2)^\pm$.
 \begin{itemize}
  \item The closure in $\mathbf{P}^+(\R^{1,2})$ of
  one of the components
  $L_1^+$ (respectively $L_1^-$) of $L_1$
  meets the closure of one of the components
  $L_2^+$ (resp. $L_2^-$) of $L_2$
  along $(L_1\cap L_2)^+$
  (resp. $(L_1\cap L_2)^-$).
  The timelike geodesics
  $L_1^\pm$ and $L_2^\pm$ of $\dS$
  are said \emph{future asymptotic}
  (resp. \emph{past asymptotic}).
  \item The closure of
  $L_1^+$ does not meet $L_2^-$
  (resp. of $L_1^-$ does not meet $L_2^+$).
 \end{itemize}
\item \emph{$L_1\cap L_2$ is spacelike},
and we denote by $(L_1\cap L_2)^\pm$
its two half-lines.\footnote{Observe that the sign
$\pm$ is here arbitrary
and has no meaning of time-orientation.}
\begin{itemize}
  \item One of the components
  $L_1^+$ (respectively $L_1^-$)
  meets the corresponding
  $L_2^+$ (resp. $L_2^-$)
  in $\dS$
  along $(L_1\cap L_2)^+$
  (resp. $(L_1\cap L_2)^-$).
  \item The closure of
  $L_1^+$ does not meet $L_2^-$
  (resp. of $L_1^-$ does not meet $L_2^+$).
 \end{itemize}
\item \emph{$L_1\cap L_2$ is timelike}.
Then the four connected component's closures
of $L_1$ and $L_2$ are pairwise disjoint.
\end{enumerate}

In hyperbolic surfaces,
the uniqueness
of geodesics
in their free homotopy class
follows from the classical fact
that two geodesics of the hyperbolic plane
which remain at a bounded distance
must be equal.
In the de-Sitter space $\dS$,
we will replace this argument
by a phenomenon of \emph{spacelike connectedness}
between asymptotic timelike geodesics.
\begin{lemma}
\label{lemma-timelikegeodesicdSfuturasymptotic}
 Let $l_1$, $l_2$ be two
 future timelike geodesics of $\dS$
 such that for any $p_i\in l_i$,
 there exists a spacelike geodesic segment joining
 $\intervallefo{p_1}{+\infty}_{l_1}$ and
 $\intervallefo{p_2}{+\infty}_{l_2}$.
 Then $l_1$ and $l_2$ are
 future asymptotic.
 The obvious analogous claim holds in the past.
\end{lemma}
\begin{proof}
Let us denote by $L_i$ the timelike projective
line of $\mathbf{P}^+(\R^{1,2})$
containing $l_i$.
We observe first that $L_1\neq L_2$.
Indeed if $L_1=L_2$ but $l_1\neq l_2$
by contradiction,
then with $s_0\subset\dS$
a spacelike geodesic intersecting each $l_i$
at some point $p_i$,
one notice that
any spacelike geodesic $s_1$
intersecting
$\intervalleoo{p_1}{+\infty}_{l_1}$
would not intersect
$\intervalleoo{p_2}{+\infty}_{l_2}$,
which would contradict our assumption.
We now reason by contraposition,
assuming that $l_1$ and
$l_2$ are \emph{not}
future asymptotic.
Since $L_1\neq L_2$,
there exists a lightlike projective line
$L$ of $\mathbf{P}^+(\R^{1,2})$
which intersects $l_1$ and $l_2$
at respective points $p_1$ and $p_2$.
One check then easily that
any spacelike geodesic $s\subset\dS$
which intersects $\intervalleoo{p_1}{+\infty}_{l_1}$
will not intersect
$\intervalleoo{p_2}{+\infty}_{l_2}$.
This concludes the proof of the contrapositive of the lemma.
\end{proof}

\subsubsection{Proof of Proposition \ref{theoremintro-existenceuniquenessgeodesics}}
\label{subsubsection-proofuniqueness}
We first show the uniqueness
of closed geodesics in primitive homotopy classes.
\begin{lemma}
\label{lemme-uniquenesssimpleclosedgeodesics}
Let $\gamma_1$ and $\gamma_2$ be two
definite closed geodesics of a class A
de-Sitter torus $T$
with a single singularity,
having primitive free homotopy classes
in $\piun{T}$.
If $\gamma_1$ and $\gamma_2$ are freely homotopic
in $T$,
then $\gamma_1=\gamma_2$ as unparametrized geodesics.
\end{lemma}
\begin{proof}
Denoting by $\zero\in T$ the singularity,
the fundamental group of
$T_*\coloneqq T\setminus\{\zero\}$
(based at any point)
is a free group on two generators,
and the inclusion $T_*\subset T$
induces a morphism
$\psi\colon\piun{T_*}\to\piun{T}\simeq\Z^2$
which is the abelianization.
The homotopy classes of simple closed curves
in $T_*$
are exactly the \emph{primitive} elements
of $\piun{T_*}$
(namely the ones that can be completed
to form a basis
of the free group $\piun{T_*}$),
and two primitive homotopy classes
of $\piun{T_*}$
belong to the same free homotopy class of curves
in $T_*$
if and only if
they are conjugated in $\piun{T_*}$.
Since $\gamma_1$ and $\gamma_2$ are simple closed curves
according to Lemma \ref{lemma-simpletimelikegeodesics},
their homotopy classes are primitive
in $\piun{T_*}$.
Since they are freely homotopic in $T$
by assumption, they have moreover the same image
by the abelianization $\psi$.
A Theorem of Nielsen \cite{nielsen_isomorphismen_1917}
shows then that the homotopy classes of
$\gamma_1$ and $\gamma_2$ in $T_*$
are conjugated
(see \cite[Corollary 3.2 p.20]{osbornePrimitivesFreeGroup1981}),
\emph{i.e.} that $\gamma_1$ and $\gamma_2$
are freely homotopic in $T_*$.
\par As in the proof of Lemma \ref{lemma-simpletimelikegeodesics},
we can assume that $\gamma_1$ and $\gamma_2$ are
timelike to fix ideas,
up to inverting the metric.
Since $\gamma_1$ and $\gamma_2$
are freely homotopic in $T_*$,
they are contained in
a closed annulus $A\subset T_*$ bounded by two
closed timelike curves
(freely homotopic to $\gamma_1$ and $\gamma_2$).
The proof of the existence Theorem
of closed definite geodesics
shows the existence of a spacelike
geodesic segment $\eta\subset A$
joining $\gamma_1$ and $\gamma_2$.
More precisely,
\cite[Lemma A.15, Proposition A.16
and Theorem A.17]{mion-moutonRigiditySingularDeSitter2024}
show the existence
of such a spacelike
locally maximizing segment $\eta\subset A$,
and
\cite[Proposition A.10]{mion-moutonRigiditySingularDeSitter2024}
shows that $\eta$ is a geodesic segment
(since it does not contain any singularity).
We fix a connected component
$\tilde{A}\subset\tilde{T}_*$
of the lift of the annulus $A\subset T_*$
in the universal cover $\tilde{T}_*$.
This topological band contains one lift
$\tilde{\gamma}_i$ of each of the two
closed timelike geodesics $\gamma_i$.
Let $\delta\colon\tilde{T}_*\to\dS$
be a developing map of the $\dS$-structure of $T_*$.
The anti-Riemannian metrics of
$\gamma_1$ and $\gamma_2$ coming
from the metric of $T_*$
are complete since the $\gamma_i$'s are compact,
and $\tilde{\gamma}_i$ is therefore complete as well.
The restriction of $\delta$ is a
local isometry from the complete
timelike geodesic
$\tilde{\gamma}_i$ to a
timelike geodesic $l_i$ of $\dS$,
and such a map has to be surjective onto $l_i$.
Indeed for any $p\in\tilde{\gamma}_i$ and any
$q'\in l_i$ in the future of $\delta(p)$,
the distance $r$ from $\delta(p)$ to $q'$
along $l_i$ is finite.
There exists thus by completeness of
$\tilde{\gamma}_i$ a point $q\in\tilde{\gamma}_i$
at distance $r$ in the future of $p$, and
we have then $\delta(q)=q'$
which shows that $\delta(\tilde{\gamma}_i)=l_i$.
\par Let $p'_1=\delta(p_1)\in l_1$ and $p'_2=\delta(p_2)\in l_2$
be any two points.
There exists then a lift $\tilde{\eta}\subset\tilde{A}$
of $\eta$ going from
the future interval
$\intervallefo{p_1}{+\infty}_{\tilde{\gamma}_1}$
of $\tilde{\gamma}_1$ starting at $p_1$
to the future interval
$\intervallefo{p_2}{+\infty}_{\tilde{\gamma}_2}$
of $\tilde{\gamma}_2$ starting at $p_2$.
The image of $\tilde{\eta}$ by $\delta$ is a spacelike geodesic
segment of $\dS$ going from
$\intervallefo{p'_1}{+\infty}_{l_1}$ to
$\intervallefo{p'_2}{+\infty}_{l_2}$.
Since such a geodesic segment exists for any two
$p'_1\in l_1$ and $p'_2\in l_2$,
Lemma \ref{lemma-timelikegeodesicdSfuturasymptotic}
shows that $l_1$ and $l_2$
are future asymptotic.
The obvious analogous reasoning shows that
$l_1$ and $l_2$
are past asymptotic,
and therefore that $l_1=l_2$ since the
future and past limit points of a timelike geodesic
of $\dS$ entirely determine it.
\par Since the de-Sitter structure of $T$
is class A, $T$ is isometric to a singular de-Sitter
torus $\mathcal{T}_{\theta,x,y}$
described in Figure \ref{figure-gluingLshaped}
according to
\cite[Theorem 9.6.(1)]{mion-moutonRigiditySingularDeSitter2024}.
The lightlike polygon $\mathcal{L}_{\theta,x,y}$
being a fundamental domain for
$T$, its interior lifts
in $\tilde{T}_*$
to an open set
$\widetilde{\mathcal{L}}$ on which
the developing map $\delta$ is injective.
We can choose a lift $\widetilde{\mathcal{L}}$
intersecting the band $\tilde{A}$,
hence containing open intervals
of $\tilde{\gamma}_1$ and $\tilde{\gamma}_2$.
Since we have shown that
$\delta(\tilde{\gamma}_1)=\delta(\tilde{\gamma}_2)$
and $\delta\restreinta_{\widetilde{\mathcal{L}}}$
is injective,
this shows that
$\tilde{\gamma}_1$ and $\tilde{\gamma}_2$
admit a common open interval,
and therefore that
$\tilde{\gamma}_1=\tilde{\gamma}_2$.
Hence $\gamma_1=\gamma_2$,
which concludes the proof of the lemma.
\end{proof}

We can now conclude the
proof of Proposition \ref{theoremintro-existenceuniquenessgeodesics}.
\begin{proof}[Proof of Proposition \ref{theoremintro-existenceuniquenessgeodesics}]
As previously,
we can assume that $\gamma$ is timelike to fix ideas.
If the timelike free homotopy class $c$ is primitive,
the claim is a consequence of
the existence Theorem \ref{theorem-existence},
and of Lemmas \ref{lemma-simpletimelikegeodesics}
and \ref{lemme-uniquenesssimpleclosedgeodesics}.
Assume now that $c=nc_0$ with $c_0$ a primitive class and $n\geq 2$.
The existence being given by the primitive case,
it suffices to show that any geodesic $\gamma$
in the free homotopy class $c$ has to be a multiple.
Indeed, such a $\gamma$ is then of the form $\gamma=n\gamma_0$
with $\gamma_0$ in the primitive class $c_0$,
and the uniqueness of $\gamma$ in the class $c$
follows thus from the one
of $\gamma_0$ in the primitive class $c_0$
(which was already proved).
Let us assume by contradiction that $\gamma$
is not a multiple.
It is then a (general position) closed geodesic
with excess self-intersection points in
its free homotopy class,
and $\gamma$ must thus contain a singular
$1$-gon or $2$-gon
according to
\cite[Theorem 4.2]{hassIntersectionsCurvesSurfaces1985}.
Since a singular $1$-gon would give a
homotopically trivial
timelike closed curve
forbidden by
\cite[Corollary A.6]{mion-moutonRigiditySingularDeSitter2024},
$\gamma$ eventually contains a singular $2$-gon.
As in the proof of Lemma \ref{lemma-simpletimelikegeodesics},
this situation leads to two distinct timelike geodesics
in $\dStilde$ having two intersection points,
which is impossible.
This contradiction concludes the proof of the
theorem.
\end{proof}

\section{Length-twist coordinates
on the deformation space of singular de-Sitter tori}
\label{section-FenchenNielsencoordinates}
In this section,
we define coordinates in the \emph{deformation space} $\Deftheta^A$
which are analogous to the classical Fenchel-Nielsen coordinates
in the Teichmüller space.
\par We denote by
$\Deftheta$
(respectively $\Deftheta^A$)
the space
of
(resp. class A)
singular de-Sitter
structures on $\Tn{2}$ of fixed area $\theta>0$
and with a unique singular point
at $\zero\in\Tn{2}$,
quotiented by the group
$\Homeo^0(\Tn{2},\zero)$
of homeomorphisms of $\Tn{2}$ isotopic to the
identity relative to $\zero$
(see \cite[\S 6.1]{mion-moutonRigiditySingularDeSitter2024}
for more details on $\Deftheta$ and its topology).
We will denote by $\m$ a singular de-Sitter structure
and by $\mzero\in\Deftheta$ its isotopy class.
The quotient $\PMod(\Tn{2},\zero)$
of homeomorphisms of $\Tn{2}$ fixing $\zero$
by $\Homeo^0(\Tn{2},\zero)$ is called the
\emph{modular group of $(\Tn{2},\zero)$},
and naturally acts on $\Deftheta$.
\par We will denote by
$L(\gamma)$
the length of any causal curve $\gamma$
in a $\dS$-surface
(the context preventing
any confusion regarding the surface
under consideration).

\subsection{de-Sitter tori as identification
space of
rectangles with two timelike edges}
\label{subsection-identificationrectangle}
As in the hyperbolic case, the first step
is to uniquely characterize a polygon
in the de-Sitter space
by the length of its sides.
\begin{lemma}
 \label{lemma-uniqueretanglelengthboundary}
For any $(k,l)\in(\R_+^*)^2$, there exists
up to action of $\Bisozero{1}{2}$
a unique rectangle
$\mathcal{R}_{\theta,k,l}$ in $\dS$
satisfying the following conditions.
\begin{itemize}
 \item The oriented boundary
 of $\mathcal{R}_{\theta,k,l}$
 is formed of four geodesic segments which are consecutively
 future $\alpha$ lightlike,
 future timelike,
 past $\alpha$ lightlike,
 and past timelike.
 \item The left (respectively right)
 timelike edge has length $k$
 (resp. $l$).
\end{itemize}
\end{lemma}
\begin{proof}
Let $\mathcal{R}_1$ and $\mathcal{R}_2$ be two such rectangles.
We name the vertices and edges of our rectangles
as indicated on the reference rectangle
$\mathcal{R}_{\theta,k,l}$ on the left-hand side of
Figure \ref{figure-rectanglethetal}.
Since $\Bisozero{1}{2}$ acts transitively on $\dS$,
we can assume without loss of generality that
$\mathcal{R}_1$ and $\mathcal{R}_2$ have the same
the bottom left vertex $x\in\dS$.
Since $\Stab_{\Bisozero{1}{2}}(x)$
moreover acts simply transitively on timelike directions
at $x$, we can also assume without loss of generality that
the left timelike eddge $J_1$ and $J_2$ of
$\mathcal{R}_1$ and $\mathcal{R}_2$
have the same direction at $x$.
Since those edges have by assumption
the same length $k$, we eventually
have $J_1=J_2=J$.
If the open right timelike edges $\Int(I_1)$
and $\Int(I_2)$
were disjoint, then
$\mathcal{R}_1$ and $\mathcal{R}_2$ would not have
the same area, contradicting our assumption.
Therefore, $\Int(I_1)$ and $\Int(I_2)$
intersect at a point $z$ indicated
on the right-hand side of
Figure \ref{figure-rectanglethetal}.
The $\beta$-leaf of $z$
intersects the $\alpha$-leaf of $y'_1$
(respectively of $y_1$)
at a point $z_+$ (resp. $z_-$).
The point $z$ defines together with
$y'_1$ and $y'_2$ (resp. with
$y_1$ and $y_2$) a triangle $T_2$
(resp. $T_1$).
The rectangles $\mathcal{R}_1$ and $\mathcal{R}_2$
are the respective union
(disjoint outside of edges)
of a common subrectangle
$\mathcal{R}_0$ and of the triangles
$T_1$ and $T_2$.
Since $\mathcal{R}_1$ and $\mathcal{R}_2$ have
the same area,
$T_1$ and $T_2$ have thus the same area.
There exists a unique orientation-preserving isometry
$\varphi$ of $\dS$
which reverses the time-orientation,
fixes $z$ and preserves
each non-oriented geodesic through $z$.
Since the angles
$\theta$ between the timelike edges
of $T_1$ and $\varphi(T_2)$
are equal
and their only non-timelike edge
is $\alpha$ lightlike,
one of the two triangles
$T_1$ and $\varphi(T_2)$
is contained in the other.
These triangles having the same
area, this shows that
$T_1=\varphi(T_2)$,
\emph{i.e.} that the
right-hand side of
Figure \ref{figure-rectanglethetal}
has a central symmetry with respect to $z$.
Since $I_1$ and $I_2$ have the same length
by assumption, we have thus
$L(\intervalleff{z}{y_1}_{I_1})
=\frac{L(I_1)}{2}
=\frac{L(I_2)}{2}
=L(\intervalleff{z}{y_2}_{I_2})$
(with $\intervalleff{p}{q}_I$
the segment from $p$ to $q$
on an oriented geodesic $I$),
and therefore
$y_1=y_2$.
Let indeed denote by $a^\theta$ the unique isometry
fixing $z$ and sending
the direction $\intervalleff{z}{y_2}_{I_2}$ on
the direction $\intervalleff{z}{y_1}_{I_1}$.
Then $a^\theta(y_2)=y_1$
since $L(\intervalleff{z}{y_1}_{I_1})
=L(\intervalleff{z}{y_2}_{I_2})$,
hence $a^\theta$ fixes $z_-\in\Fbeta(z)\setminus\{z\}$.
Since $\Stab_{\Bisozero{1}{2}}(z)$ acts simply
transitively on
$\Fbeta(z)\setminus\{z\}$,
this shows that $a^\theta=\id$
hence that $y_1=y_2$.
Therefore $y'_1=y'_2$ as well,
which shows that
$\mathcal{R}_1=\mathcal{R}_2$
and concludes the proof.
\end{proof}
\begin{figure}[!h]
	\begin{center}
		\def\svgwidth{1.16 \columnwidth}
			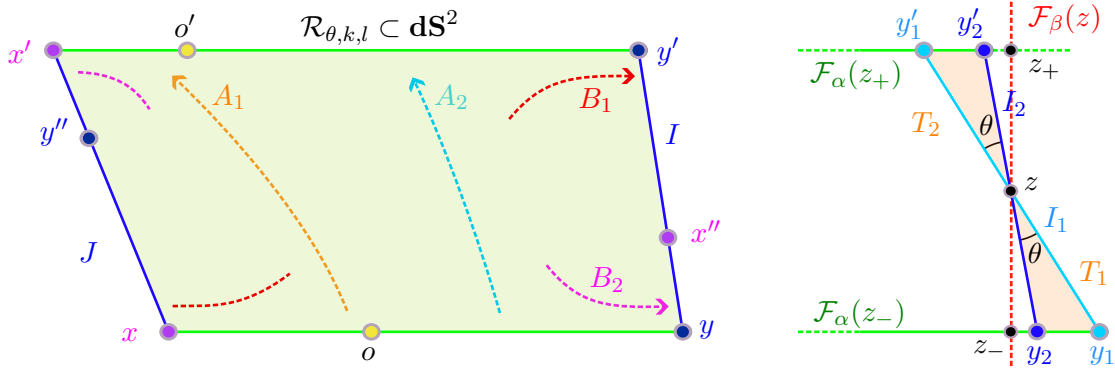
	\end{center}
	\caption{Rectangle
	$\mathcal{R}_{\theta,k,l}$ of $\dS$
	with area $\theta$,
	$\alpha$ lightlike horizontal edges
	and timelike vertical edges of lengths $(k,l)$.}
			\label{figure-rectanglethetal}
\end{figure}

We will denote
$\mathcal{R}_{\theta,l}\coloneqq\mathcal{R}_{\theta,l,l}$
when both timelike sides of the rectangle
are equal.
\begin{remark}
 \label{remark-identificationspace}
The choice of points $(y'',x'',o,o')$
on the boundary of $\mathcal{R}_{\theta,l}$
defines
identifications of the resulting
eight edges of $\partial\mathcal{R}_{\theta,l}$
by isometries
$(A_1,A_2,B_1,B_2)$
of $\Bisozero{1}{2}$,
as indicated in Figure
\ref{figure-rectanglethetal}.
The quotient of $\mathcal{R}_{\theta,l}$
by such identifications
is a $\dS$-torus
of area $\theta$ and
with at most three singular points
$(y,x,o)$,
henceforth
called an \emph{identification space}
of $\mathcal{R}_{\theta,l}$.
We can analogously construct
a de-Sitter annulus with
(piecewise) geodesic timelike boundary
as an \emph{identification space
of the lightlike edges of
$\mathcal{R}_{\theta,k,l}$}.
\end{remark}
The de-Sitter annuli
with timelike geodesic boundary
and containing a single singularity
are basic building blocks
for singular de-Sitter tori.
\begin{lemma}
 \label{lemma-decoupagetoredeSitter}
Let $\gamma_1$ and $\gamma_2$
be two
simple closed timelike
geodesics of a class A singular de-Sitter torus $T$.
\begin{enumerate}
 \item If $\gamma_1=\gamma_2$
 and $T$ contains a unique singularity,
 then $T$ is isometric to an identification space
 of the rectangle $\mathcal{R}_{\theta,l}$,
 with $l$ the length
 of $\gamma_1=\gamma_2$ and $\theta$ the area of $T$.
 \item Assume that $\gamma_1$ and $\gamma_2$
are freely homotopic,
disjoint and
bound an open annulus $A$
of area $\theta$
containing a single singularity.
Then the singular de-Sitter sub-annulus
$\Cl(A)\subset T$
with timelike geodesic boundary
is isometric to an identification space
of the lightlike edges of
$\mathcal{R}_{\theta,l_1,l_2}$,
with $l_i$ the length
 of $\gamma_i$.
\end{enumerate}
\end{lemma}
\begin{proof}
We consider the case (2), the proof of (1) being analogous.
Let $p$ be the unique singularity of $A$.
Since $T$ is class A, the
$\alpha$ leaf of $p$
describes an arc $\delta$ in $A$ joining $\gamma_1$ to $\gamma_2$.
Let $R$ be the de-Sitter surface with geodesic
boundary obtained by cuting $A$ along
$\delta$.
Since the interior and the
timelike boundaries of $A$
do not contain any singularity,
the same proof \emph{mutatis mutandis} than
the rectangular case in
\cite[Proposition 9.2]{mion-moutonRigiditySingularDeSitter2024}
shows that $R$ is isometric
to the rectangle $\mathcal{R}_{\theta,l_1,l_2}\subset\dS$,
and that $A$ is an identification space
of $\mathcal{R}_{\theta,l_1,l_2}$.
\end{proof}

\subsection{Length-twist
coordinates in $\Deftheta$}
\label{subsection-lengthtwist}
Let $(\azero,\bzero)\in(\piun{\Tn{2}}_*)^2$
be a fixed pair of primitive homotopy classes
of algebraic intersection number
$\hat{i}(\azero,\bzero)=1$.
We can now conclude the
construction of coordinates in the
open subset
\begin{equation*}
 \label{equation-Defthetaa}
 \Deftheta^A_{\azero}\coloneqq
 \enstq{\mzero\in\Deftheta}{
 \azero\in\mathcal{C}^\mzero}
\end{equation*}
of class A structures
of which $\azero$ is a timelike class.
Proposition \ref{theoremintro-existenceuniquenessgeodesics}
already gave us a
\emph{length function of $\azero$}
on $\Deftheta^A_{\azero}$, defined as
\begin{equation*}
 \label{equation-lengthfunctionintro}
 \mathcal{L}_\azero\colon\mzero\in\Deftheta^A_{\azero}
 \mapsto
 L^\m(\gamma_\azero^\m)\in\R_+^*
\end{equation*}
with $\gamma_\azero^\m$ the unique
simple closed timelike geodesic of $\m$
in the free homotopy class $\azero$.
Since the lightlike foliations
of a class A structure are suspensions,
the $\alpha$-lightlike leaf of the
singularity $\zero$ has a first intersection
with the
transverse closed curve
$\gamma_\azero^\m$
in the future (respectively in the past)
at a point denoted by $x_+$
(resp. $x_-$).
We refer to
\cite[Proof of Lemma 9.9]{mion-moutonRigiditySingularDeSitter2024}
for more details on this claim.
The concatenation of the $\alpha$-lightlike
segment $\intervalleff{x_+}{x_-}_\alpha$
and of the future segment
$\intervalleff{x_-}{x_+}_{\gamma_\azero^\m}$
on the oriented geodesic $\gamma_\azero^\m$ is a closed curve,
which can be written as
\begin{equation*}
 \label{equation-nab}
 \intervalleff{x_+}{x_-}_\alpha\cdot
 \intervalleff{x_-}{x_+}_{\gamma_\azero^\m}
 =\bzero+n_{\azero,\bzero}(\mzero)\cdot\azero
\end{equation*}
for a unique $n_{\azero,\bzero}(\mzero)\in\Z$
(depending only on the isotopy
class $\mzero$).
We adopt the convention that the segment
$\intervalleoo{x_-}{x_+}_{\gamma_\azero^\m}$
is empty if $x_-=x_+$
(\emph{i.e.} if the $\alpha$-lightlike leaf of $\zero$ is closed).
\begin{definition}
 \label{definition-Thetaab}
 The \emph{twist coordinate around $\azero$
 with longitude $\bzero$} of
 $\mzero\in\Deftheta^A_\azero$ is
 the real number
 \begin{equation*}
  \label{equation-Thetaab}
  \Theta_{\azero,\bzero}(\mzero)
 \coloneqq
 n_{\azero,\bzero}(\mzero)+
 \frac{L^\m\left(\intervalleoo{x_-}{x_+}_{\gamma_\azero^\m}\right)}{\mathcal{L}_\azero(\mzero)}
 \in\R.
 \end{equation*}
\end{definition}

We can now conclude the proof
of Theorem
\ref{theoremeintro-fenchennielsen}.
\begin{proof}[Proof
of Theorem
\ref{theoremeintro-fenchennielsen}]
We recall from Lemma \ref{lemma-uniqueretanglelengthboundary}
that the rectangle $\mathcal{R}_{\theta,l}$
is unique up to the action of $\Bisozero{1}{2}$,
and we henceforth fix once and for all such a rectangle.
The points $(x,y,x',y')$ are thus fixed.
Since $\Bisozero{1}{2}$ acts simply transitively
on future timelike segments of the same
finite length of $\dS$,
and also on proper
future $\alpha$ lightlike segments
of $\dS$,
the choice of points $(y'',x'',o,o')$
on $\partial\mathcal{R}_{\theta,l}$
entirely determines the identifications
$(A_1,A_2,B_1,B_2)$
of the Figure \ref{figure-rectanglethetal}.
The timelike edges
$\intervalleff{x}{y''}_J$ and $\intervalleff{x''}{y'}_I$
having moreover the same length,
$x''$ is actually determined by $y''$.
In the end, the three points $(y'',o,o')$
on $\partial\mathcal{R}_{\theta,l}$ entirely determine
a $\dS$-torus
$\mathcal{T}_{\theta,l}(y'',o,o')$
with at most three singularities at the respective projections
$x$, $y$ and $o$ of the marked points
on the boundary (named by the same letters
on Figure \ref{figure-rectanglethetal}).
The holonomy of a small
positively oriented closed curve
$\gamma_{x/y/o}$ around one of
those points is respectively equal to:
\begin{equation*}
 \label{equation-holonomyxyo}
 \hol(\gamma_x)=A_1^{1}B_2^{-1}B_1,
 \hol(\gamma_y)=B_2B_1^{-1}A_2
 \text{~and~}
 \hol(\gamma_o)=A_2^{-1}A_1.
\end{equation*}
We refer to \cite[Proposition 4.1]{mion-moutonRigiditySingularDeSitter2024}
for more details concerning the computation
of these holonomies.
Therefore, $x$ and $y$ are regular points
of $\mathcal{T}_{\theta,l}(y'',o,o')$
if and only if
\begin{equation}
 \label{equation-A1A2}
 A_1=B_2^{-1}B_1
 \text{~and~}
 A_2=B_1B_2^{-1}
\end{equation}
according to Lemma \ref{lemma-dSstructuresholonomy}.
In other words,
under the condition of having a unique singular
point at $o\in\mathcal{T}_{\theta,l}(y'',o,o')$,
$y''\in\partial\mathcal{R}_{\theta,l}$
entirely determines $(B_1,B_2)$ hence
$(A_1,A_2)$ according to \eqref{equation-A1A2},
which in return determines $(o,o')$.
\par Let $\gamma$ denote the timelike simple geodesic loop
of $\mathcal{T}_{\theta,l}(y'',o,o')$ defined
by the projection of a timelike edge
($I$ or $J$)
of $\mathcal{R}_{\theta,l}$.
We have shown so far that
$y''\in\partial\mathcal{R}_{\theta,l}$
determines a unique
$\dS$-torus
$\mathcal{T}_{\theta,l,y''}$ of area $\theta$
for which $\gamma$ has no singularity,
\emph{i.e.} is a simple closed geodesic.
Moreover, $\gamma$ has length $l$
and $\mathcal{T}_{\theta,l,y''}$
has a unique singularity at $o$.
Let $a$ denote the free homotopy class of $\gamma$,
and $b$ denote the free homotopy class of the simple closed
curve
of $\mathcal{T}_{\theta,l,y''}$
obtained by projecting a simple arc of
$\mathcal{R}_{\theta,l}$
intersecting its boundary only at $y$ and $y''$
and oriented from $I$ to $J$.
There exists, up to pre-composition
by a homeomorphism of $\Tn{2}$ isotopic to the identity
relative to $\zero$,
a unique homeomorphism
$\Phi\colon\Tn{2}\to\mathcal{T}_{\theta,l,y''}$
sending $\zero$ on $o$ and
the homotopy basis $(\azero,\bzero)$
on $(a,b)$.
This well-known fact is due to a result of Epstein
\cite{epsteinCurves2manifoldsIsotopies1966}
(see also \cite[Proposition 1.6, Theorem 2]{beguinFixedPointSets2020}
for more details).
We henceforth denote by
$\mzero_{\theta,l,y''}\in\Deftheta^A$
the pull-back of the $\dS$-structure of $\mathcal{T}_{\theta,l,y''}$
by $\Phi$
on $\Tn{2}$.
Let $D_\azero\in\PMod(\Tn{2},\zero)$
denote the \emph{positive Dehn twist}
around $\azero$, namely the unique mapping class
satisfying $D_\azero(\azero)=\azero$ and
$D_\azero(\bzero)=\bzero+\azero$.
Then for any $n\in\Z$,
the length-twist coordinates of
$(D_\azero^n)_*\mzero_{\theta,l,y''}\in\Deftheta^A$
are given by
\begin{equation}
 \label{equation-lengthtwistg}
 \mathcal{L}_\azero((D_\azero^n)_*\mzero_{\theta,l,y''})=l
 \text{~and~}
 \Theta_{\azero,\bzero}((D_\azero^n)_*\mzero_{\theta,l,y''})=
 n+\frac{L\left(\intervalleff{x}{y''}_J\right)}{l},
\end{equation}
by the very definition of $\Phi$.
This shows the surjectivity of the length-twist map
\begin{equation*}
 \label{equation-lengthtwistproof}
 \mathcal{L}_\azero\times\Theta_{\azero,\bzero}
\colon\Deftheta^A_\azero\to\R^*_+\times\R.
\end{equation*}
\par Let now $\mzero_1$ and $\mzero_2$
be two points of $\Deftheta^A_\azero$ having the same length-twist
coordinates $\mathcal{L}_\azero(\mzero_i)=l$
and $\Theta_{\azero,\bzero}(\mzero_i)=n+u$,
with $n\in\Z$ and $u\in\intervalleff{0}{1}$.
Lemma \ref{lemma-decoupagetoredeSitter} shows that
$(D_\azero^{-n})_*\mzero_1$ and $(D_\azero^{-n})_*\mzero_2$
are respectively isometric to
structures $\mzero_{\theta,l,y_1}$ and $\mzero_{\theta,l,y_2}$,
and \eqref{equation-lengthtwistg} shows that
$y_1=y_2$ since
$\Theta_{\azero,\bzero}(\mzero_1)=\Theta_{\azero,\bzero}(\mzero_2)$.
This shows that $\mzero_1=\mzero_2$, hence the injectivity
of $\mathcal{L}_\azero\times\Theta_{\azero,\bzero}$.
This map being clearly continuous,
it remains to show that its inverse is continuous as well.
\par We fix the bottom-left vertex $x$
of $\mathcal{R}_{\theta,l}$
as well as the future timelike half-geodesic
$d\subset\dS$ containing $J$.
Let $p_r$ denote the unique point of $d$ at distance
$r$ from $x$.
For any $l\in\R^*_+$, there exists
a unique rectangle $\mathcal{R}_{\theta,l}$
of area $\theta$ having
$\intervalleff{x}{p_l}_d$ as left timelike edge $J$.
The map
\begin{equation*}
 \label{equation-inverseLTheta}
(l,u)\in\R^*_+\times\intervallefo{0}{1}
\mapsto\mzero_{\theta,l,p_{ul}}\in\Deftheta^A_\azero
\end{equation*}
is continuous and is a local inverse of
$\mathcal{L}_\azero\times\Theta_{\azero,\bzero}$,
which concludes the proof.
\end{proof}

Theorem \ref{theoremeintro-fenchennielsen}
gives a new proof of the following result,
proved in \cite[Theorem E]{mion-moutonRigiditySingularDeSitter2024}
by using ``lightlike coordinates''.
\begin{corollary}
 \label{corollary-DefthetaAhausdorfftopologicalsurface}
 $\Deftheta^A$ is a Hausdorff topological surface.
\end{corollary}
\begin{proof}
 Since $\Deftheta^A$ is covered by the
 $\Deftheta^A_\azero$ for $\azero\in\piun{\Tn{2}}_*$,
 Theorem \ref{theoremeintro-fenchennielsen}
 shows that the length-twist coordinates
 provide a topological atlas of $\Deftheta^A$.
 Each $\Deftheta^A_\azero$ is moreover Hausdorff
 as it is globally homeomorphic to $\R^*_+\times\R$.
 Let $\mzero_1\neq\mzero_2\in\Deftheta^A$.
 If
 $\mathcal{C}^{\mzero_1}\cap\mathcal{C}^{\mzero_2}\neq\varnothing$,
 $\mzero_1$ and $\mzero_2$
 belong to a common $\Deftheta^A_\azero$ in which they
 can thus be separated
 by open subsets.
 If
 $\mathcal{C}^{\mzero_1}\cap\mathcal{C}^{\mzero_2}=\varnothing$,
 then the pairs of oriented projective asymptotic cycles
 $\mathcal{A}(\mzero_1)\coloneqq
 (A^+(\Falpha^{\mzero_1}),A^+(\Fbeta^{\mzero_1}))$
 and $\mathcal{A}(\mzero_2)$
 are distinct.
 The asymptotic cycle map
 \begin{equation*}
  \label{equation-asymptoticcycle}
  \mathcal{A}\colon\Deftheta\to
  (\mathbf{P}^+(\Homologie{1}{\R}{\Tn{2}}))^2
 \end{equation*}
 being continuous according to
 \cite[Lemma 6.5]{mion-moutonRigiditySingularDeSitter2024},
 $\mzero_1$ and $\mzero_2$ can then be separated in $\Deftheta$
 by the pre-images of open subsets
 separating $\mathcal{A}(\mzero_1)$ and $\mathcal{A}(\mzero_2)$.
 This shows that $\Deftheta^A$
 is Hausdorff and concludes the proof.
\end{proof}

\begin{remark}
 \label{remark-coordonneesdiff}
 The coordinates
 $\mu_{\theta,x,y}$ constructed on $\Deftheta^A$
 from lightlike informations
 in \cite[Lemma 9.4]{mion-moutonRigiditySingularDeSitter2024}
 (see also
 \cite[Figure 4.2 and \S 6.2]{mion-moutonRigiditySingularDeSitter2024})
 and the length-twist coordinates
 $\mathcal{L}_\azero\times\Theta_{\azero,\bzero}$
 that we just constructed,
 have real-analytical coordinate changes.
 Consequently, both of these
 atlases define the same natural
 real-analytic structure on $\Deftheta^A$.
 \par The unique simple closed timelike geodesic
 of $\mathcal{T}_{\theta,x,y}$ in a given
 timelike free homotopy class
 $c$ is indeed the projection of a disjoint union
 of timelike arcs $\delta_c$
 in $\mathcal{L}_{\theta,x,y}$,
 whose finite number is locally constant in $c$.
 One may now observe that these arcs vary
 real-analytically in $(x,y)$.
 This shows that both the length
 and the twist coordinates are real-analytic
 in $(x,y)$.
\end{remark}

\section{Rigidity of two timelike lengths}
\label{section-rigiditytimelikeMLS}
We prove in this section Theorem
\ref{theoremintro-rigidityMLS}.
Let $\m_1$ and $\m_2$ be two class A
 de-Sitter structures
 of $\Tn{2}$
 of equal area $\theta$ and
 having a unique singularity,
 giving the same lengths
\begin{equation*}
 \label{equation-memelongueuerbasetemps}
 (\mathcal{L}_{\azero}(\mzero_1),\mathcal{L}_{\bzero}(\mzero_1))
=(\mathcal{L}_{\azero}(\mzero_2),\mathcal{L}_{\bzero}(\mzero_2))
\eqqcolon(l,k)
 \end{equation*}
to a common basis of timelike homotopy classes
$(\azero,\bzero)$.
\par \textbf{Step 1.}
With $\gamma_\azero$
the unique geodesic
of $\m_1$
in the free
homotopy class $\azero$,
we first observe that
\emph{$\mzero_2$
is a twist of $\mzero_1$ around $\gamma_\azero$}
in the following sense.
\par Let $A$ be the singular $\dS$-annulus
with two timelike geodesic boundary components
$\gamma_\azero^\pm$ obtained
by cuting the singular $\dS$-torus
$T=(\Tn{2},\m_1)$ along $\gamma_\azero$.
We name the left (respectively right)
boundary component of $A$ by
$\gamma_\azero^-$ (resp. $\gamma_\azero^+$)
with respect to the orientation of $A$.
The definition of $A$ from $T$
comes with a unique map
$\iota\colon\gamma_\azero^+\to\gamma_\azero^-$
between the boundary components
such that the quotient $\bar{A}$ of $A$
by the identification
$p\in\gamma_\azero^+\sim \iota(p)\in\gamma_\azero^-$
identifies with the original singular $\dS$-torus $T$.
We denote by $\azero_A$
the free homotopy class
of $\gamma_\azero^-$ in $A$.
To encode the isotopy class $\mzero_1$,
we need to also
keep track of the free homotopy class
$\bzero$ in $A$.
Let $\bzero_A$ be the unique class of simple arcs
$\delta\colon\intervalleff{0}{1}\to A$
going from points $\delta(0)\in\gamma_\azero^+$
to their image $\iota(\delta(0))=\delta(1)\in\gamma_\azero^-$
modulo free homotopy fixing the base points,
such that the simple closed curve $\bar{\delta}$
obtained in $\bar{A}\equiv T$ by projecting
any arc $\delta$ in $\bzero_A$
is freely homotopic to $\bzero$.
Identifying $\gamma_\azero^-$ with $\R/l\Z$
through a unit speed parametrization,
we denote by $R_u$ the rotation
$p\in\gamma_\azero^-\mapsto p+u\in\gamma_\azero^-$
for any $u\in\intervallefo{0}{l}$.
The quotient $T_u$
of $A$ by the identification
$p\in\gamma_\azero^+\sim R_u\circ\iota(p)\in\gamma_\azero^-$
of its boundary components is a $\dS$-torus
of area $\theta$ with a unique singular point.
It is endowed with the primitive free homotopy
class $\azero_u$ induced by $\azero_A$.
The concatenation of any arc $\delta$ in the
class $\bzero_A$ with the segment
$\intervalleff{\delta(1)}{\delta(1)+u}_{\gamma_\azero^-}$
of $\gamma_\azero^-$
projects in $T_u$ to a simple closed curve
whose free homotopy
class does not depend on $\delta$ and is denoted
by $\bzero_u$.
We denote by $(\mzero_1)_u$
the point of
$\Deftheta^A$ defined by the pullback
of the singular $\dS$-structure of $T_u$ by any homeomorphism
from $\Tn{2}$ to $T_u$
sending $(\azero,\bzero)$ to $(\azero_u,\bzero_u)$
and $\zero$ to the unique singular point of $T_u$.
\par Since $\mathcal{L}_{\azero}(\mzero_1)=
\mathcal{L}_{\azero}(\mzero_2)=l$,
$\mzero_1$ and $\mzero_2$ are identification spaces
of the same rectangle $\mathcal{R}_{\theta,l}$
according to
Lemma \ref{lemma-decoupagetoredeSitter},
the timelike edges $I$ and $J$ of $\mathcal{R}_{\theta,l}$
projecting to the unique simple closed timelike geodesic
of $\m_1$ (respectively $\m_2$) in the class $\azero$.
There exists thus $n\in\Z$ and $u\in\intervallefo{0}{l}$
such that
\begin{equation*}
 \label{equation-g2g1}
 \mzero_2=(D_\azero^n)_*(\mzero_1)_{u},
\end{equation*}
with $D_\azero$ the positive Dehn twist around $\azero$.
Note that, possibly exchanging the roles of $\mzero_1$ and $\mzero_2$,
we can assume without loss of generality that $n\geq0$.
Identifying henceforth $\mzero_2$ and $(D_\azero^n)_*(\mzero_1)_{u}$,
the unique simple closed
timelike geodesic
of $\m_2$
in the free homotopy class $\bzero$
appears in $A$ as a simple timelike geodesic arc
$\delta\colon\intervalleff{0}{1}\to A$
going from
$p_+\coloneqq\delta(0)\in\gamma_\azero^+$ to
$p_-\coloneqq\delta(1)\in\gamma_\azero^-$.
Since the projection of $\delta$
in $(\Tn{2},(D_\azero^n)_*(\mzero_1)_{u})$
belongs to the free homotopy class $\bzero$,
the concatenation
\begin{equation*}
 \label{equation-deltaA}
 \delta\cdot\intervalleff{p_-}{p_-+u}_{\gamma_\azero^-}
\end{equation*}
belongs to the free homotopy class $\bzero_A-n\azero_A$
of arcs in $A$.
We also denote by $\gamma_\bzero\colon\intervalleff{0}{1}\to A$
the arc of $A$
in the class $\bzero_A$
induced by the unique simple closed timelike geodesic
of $T$ in the class $\bzero$.
Let $E$ be the universal cover of $A$ endowed
with the singular $\dS$-structure induced by $A$.
This is a singular $\dS$-surface
homeomorphic to a band
$\intervalleff{0}{1}\times\R$,
with two timelike geodesic components
$\tilde{\gamma}_\azero^-\simeq\{0\}\times\R$
and $\tilde{\gamma}_\azero^+\simeq\{1\}\times\R$
which are the universal covers of the
geodesic boundary components $\gamma_\azero^\pm$ of $A$.
Fixing a lift $\tilde{p}_+\in\tilde{\gamma}_\azero^+$
of $p_+$ in $E$,
we consider the lift $\tilde{\delta}$ of $\delta$ starting
from $\tilde{p}_+$
and denote by $\tilde{p}_-$ its endpoint.

\par \textbf{Step 2.} \emph{Assume by contradiction that $n\neq0$.}
The arc $\delta$ has then $n$ intersection points
with $\gamma_\bzero\setminus\gamma_\azero^-$,
which we denote by $(p_1,\dots,p_n)$
in the increasing order in which they
are met on $\delta$ with its future orientation.
With $\tilde{p}_i\in\tilde{\delta}$
the respective lifts of the $p_i$'s
on $\tilde{\delta}$,
let $\tilde{\gamma}_\bzero^i$ be the lift
of $\gamma_\bzero$ passing through $\tilde{p}_i$,
and $q_i\in\tilde{\gamma}_\azero^+$
be the starting point of $\tilde{\gamma}_\bzero^i$.
These constructions are illustrated
on the right-hand side of Figure \ref{figure-preuven0}
for $n=3$.
Observe that it may be that
$p_+\in\gamma_\bzero$ \emph{i.e.} that
$p_+=p_1$,
in which case $\tilde{p}_+=\tilde{p}_1=q_1$
(this situation is illustrated on the left-hand side
of Figure \ref{figure-preuven0}).
Since $\azero$ is a timelike free homotopy class of $\mzero_1$,
there exists according to
\cite[Proposition A.8.(1)]{mion-moutonRigiditySingularDeSitter2024}
a simple closed timelike curve $\sigma_i$ of $A$
passing through $p_i$
and avoiding the singularity
in the free homotopy class
$\azero_A$.\footnote{Though the latter claim is not explicitly stated
in this way,
it is a byproduct of the proof
of \cite[Proposition A.8.(1)]{mion-moutonRigiditySingularDeSitter2024}.}
Let $\tilde{\sigma}_i$ be the lift of $\sigma_i$
starting at $\tilde{p}_i$,
and $\tilde{p}'_i$ be the endpoint of $\tilde{\sigma}_i$.
\begin{figure}[!h]
	\begin{center}
		\def\svgwidth{1.16 \columnwidth}
			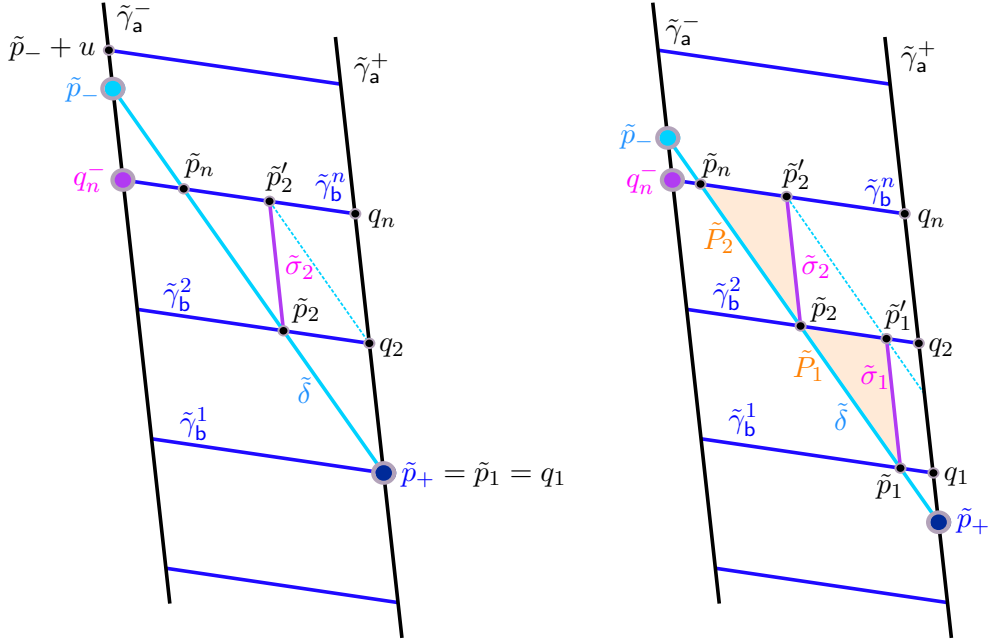
	\end{center}
	\caption{Lifts of the arcs
	$\textcolor{DodgerBlue}{\delta}$,
	$\textcolor{Blue}{\gamma_\bzero}$
	and $\gamma_\azero^\pm$
	in the universal cover $E$ of the annulus $A$
	(for $n=3$).}
			\label{figure-preuven0}
\end{figure}

\par For any $i=1,\dots,n-1$, the three future timelike segments
$\tilde{\sigma}_i$,
$\intervalleff{\tilde{p}'_i}{\tilde{p}_{i+1}}_{\tilde{\gamma}^{i+1}_\bzero}$
and
$\intervalleff{\tilde{p}_i}{\tilde{p}_{i+1}}_{\tilde{\delta}}$
bound a closed triangle $P_i\subset E$
containing at most one singularity
(in its interior),
illustrated in Figure \ref{figure-preuven0}.
There exists therefore an embedding of an open
neighbourhood $U_i$ of $\tilde{P}_i$
in $\dStilde$ or in $\dSsingularthetatilde$.
This is a consequence of
\cite[Lemma 3.5]{mion-moutonRigiditySingularDeSitter2024},
and of the fact that any class A $\dS$-torus
is isometric to a torus
$\mathcal{T}_{\theta,x,y}$ described
in Figure \ref{figure-gluingLshaped}
according to
\cite[Theorem 9.6.(1)]{mion-moutonRigiditySingularDeSitter2024}.
In particular,
any future timelike geodesic segment
in $U_i$ from $x$ to $y$
maximizes the length
of future causal curves in $U_i$ going from $x$ to $y$
according to Proposition
\cite[Proposition A.12]{mion-moutonRigiditySingularDeSitter2024}.
Since
$\tilde{\sigma}_i\cdot\intervalleff{\tilde{p}'_i}{\tilde{p}_{i+1}}_{\tilde{\gamma}^{i+1}_\bzero}$
is a future causal curve in $U_i$ from
$\tilde{p}_i$ to $\tilde{p}'_{i+1}$,
we have therefore the following reverse
triangle inequality:
\begin{equation}
 \label{equation-ingalitetrianginverse1}
 L\left(\intervalleff{\tilde{p}_i}{\tilde{p}_{i+1}}_{\tilde{\delta}}\right)
 \geq
 L\left(\intervalleff{\tilde{p}'_i}{\tilde{p}_{i+1}}_{\tilde{\gamma}^{i+1}_\bzero}\right)
 +L\left(\tilde{\sigma}_i\right)
 >L\left(\intervalleff{\tilde{p}'_i}{\tilde{p}_{i+1}}_{\tilde{\gamma}^{i+1}_\bzero}\right).
\end{equation}
The second strict inequality is due to the fact that
$L\left(\tilde{\sigma}_i\right)>0$ since it is a
non-trivial timelike curve.
Let $q_n^-\in\tilde{\gamma}_\azero^-$ denote the endpoint of
$\tilde{\gamma}^{n}_\bzero$.
Then by the same arguments as before,
the three timelike segments
$\intervalleff{\tilde{p}_n}{q_n^-}_{\tilde{\gamma}^{n}_\bzero}$,
$\intervalleff{q_n^-}{\tilde{p}_-}_{\tilde{\gamma}^-_\azero}$
and
$\intervalleff{\tilde{p}_n}{\tilde{p}_-}_{\tilde{\delta}}$
bound a closed triangle $\tilde{P}_n\subset E$
containing at most one singularity,
in which timelike geodesic segments
maximize thus the causal lengths.
Therefore
\begin{equation}
 \label{equation-ingalitetrianginverse2}
 L\left(\intervalleff{\tilde{p}_n}{\tilde{p}_-}_{\tilde{\delta}}\right)
 \geq
 L\left(\intervalleff{\tilde{p}_n}{q_n^-}_{\tilde{\gamma}^{n}_\bzero}\right)+
 L\left(\intervalleff{q_n^-}{\tilde{p}_-}_{\tilde{\gamma}^-_\azero}\right)
 >L\left(\intervalleff{\tilde{p}_n}{q_n^-}_{\tilde{\gamma}^{n}_\bzero}\right)
\end{equation}
as previously.
If $p_+\notin\gamma_\bzero$
\emph{i.e.} $\tilde{p}_1\neq\tilde{p}_+$,
we also have to consider
the triangle $\tilde{P}_0\subset E$
bounded by the timelike segments
$\intervalleff{\tilde{p}_+}{q_1}_{\tilde{\gamma}^+_\azero}$,
$\intervalleff{q_1}{\tilde{p}_1}_{\tilde{\gamma}^1_\bzero}$
and
$\intervalleff{\tilde{p}_+}{\tilde{p}_1}_{\tilde{\delta}}$
(this case is illustrated on the right-hand side
of Figure \ref{figure-preuven0}).
This
yields as previously the inequality
\begin{equation}
 \label{equation-ingalitetrianginverse3}
 L\left(\intervalleff{\tilde{p}_+}{\tilde{p}_1}_{\tilde{\delta}}\right)
 \geq
 L\left(\intervalleff{q_1}{\tilde{p}_1}_{\tilde{\gamma}^1_\bzero}\right)+
 L\left(\intervalleff{\tilde{p}_+}{q_1}_{\tilde{\gamma}^+_\azero}\right)
 >L\left(\intervalleff{q_1}{\tilde{p}_1}_{\tilde{\gamma}^1_\bzero}\right).
\end{equation}
Observe now that
$L\left(\intervalleff{\tilde{p}'_i}{\tilde{p}_{i+1}}_{\tilde{\gamma}^{i+1}_\bzero}\right)
=L\left(\intervalleff{p_i}{p_{i+1}}_{\gamma_\bzero}\right)$
for $i=1,\dots,n-1$,
$L\left(\intervalleff{\tilde{p}_n}{q_n^-}_{\tilde{\gamma}^{n}_\bzero}\right)
=L\left(\intervalleff{p_n}{\gamma_\bzero(1)}_{\gamma_\bzero}\right)$
and
$L\left(\intervalleff{q_1}{\tilde{p}_1}_{\tilde{\gamma}^1_\bzero}\right)=
L\left(\intervalleff{\gamma_\bzero(0)}{p_1}_{\gamma_\bzero}\right)$.
Therefore,
the sum
of inequalities \eqref{equation-ingalitetrianginverse1},
\eqref{equation-ingalitetrianginverse2} and
\eqref{equation-ingalitetrianginverse3}
gives
\begin{equation*}
 \label{equation-contradictionstep2}
 \mathcal{L}_\bzero(\mzero_2)=L(\delta)>
 L(\gamma_\bzero)=\mathcal{L}_\bzero(\mzero_1)
\end{equation*}
which contradicts our original assumption.
This contradiction shows that \emph{n=0}.
\par \textbf{Step 3, case A:}
\underline{$p_+\in\gamma_\bzero$},
\emph{i.e.} $\delta(0)=\gamma_\bzero(0)$.
Then since
\begin{equation}
 \label{equation-extremitesdelta}
 \delta(1)+u=\iota(\delta(0))=\gamma_\bzero(1),
\end{equation}
the timelike segments
$\tilde{\delta}$,
$\intervalleff{\tilde{p}_-}{\tilde{p}_-+u}_{\tilde{\gamma}_\azero^-}$
and $\tilde{\gamma}_\bzero$
bound in $E$ a triangle to which we can apply
our previous argument.
This shows that
$L(\tilde{\delta})\geq
L\left(\tilde{\gamma}_\bzero\right)
+L\left(\intervalleff{\tilde{p}_-}{\tilde{p}_-+u}_{\tilde{\gamma}_\azero^-}\right)$, hence that $u=0$
since $L(\delta)=L(\gamma_\bzero)$.
\par \textbf{Step 3, case B:} \underline{$p_+\notin\gamma_\bzero$}.
In this case, the arcs $\delta$
and $\gamma_\bzero$ bound together
with $\gamma_\azero^\pm$
two rectangles with timelike edges in $A$,
only of them containing the unique singularity of $A$
(in its interior).
Without lost of generality,
we assume that the closed rectangle $P\subset A$
``above $\gamma_\bzero$''
does not contain the singularity.
This situation is represented in Figure
\ref{figure-rectangleGaussBonnet}.
\begin{figure}[!h]
	\begin{center}
		\def\svgwidth{0.76 \columnwidth}
			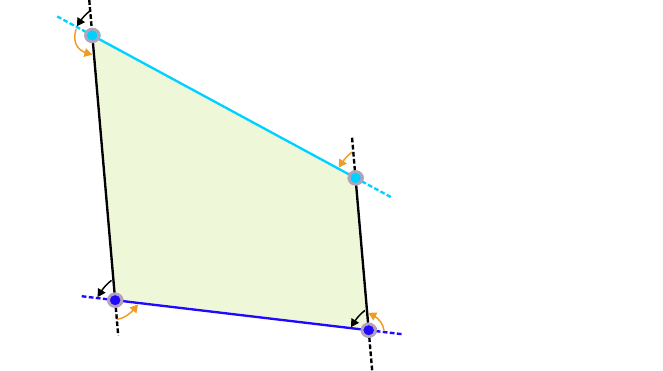
	\end{center}
	\caption{Rectangle $P\subset\dS$ bounded
	by $\textcolor{DodgerBlue}{\delta}$,
	$\textcolor{Blue}{\gamma_\bzero}$
	and $\gamma_\azero^\pm$,
	and its \textcolor{orange}{exterior angles}
	at vertices $v_i$.}
			\label{figure-rectangleGaussBonnet}
\end{figure}

We finally make use of an additional property of $\delta$
and $\gamma_\bzero$
that we had not mentioned until now.
Since $\gamma_\bzero$ derives from a
simple closed geodesic of $T$,
its angles
with $\gamma_\azero^-$ and with $\gamma_\azero^+$ are equal:
\begin{equation*}
 \label{equation-anglegammabgammaa}
 \alpha\coloneqq\angleLorentzien{\gamma_\azero^+}{\gamma_\bzero}_{\gamma_\bzero(0)}
 =\angleLorentzien{\gamma_\azero^-}{\gamma_\bzero}_{\gamma_\bzero(1)}.
\end{equation*}
In the same way,
$\delta$ projects by construction
to a simple closed geodesic of $(\mzero_1)_u=\mzero_2$.
Since the gluing
$R_u\circ\iota\colon\gamma_\azero^+\to\gamma_\azero^-$
is made by isometries,
the angles of $\delta$ with $\gamma_\azero^-$ and $\gamma_\azero^+$
are also equal:
\begin{equation*}
 \label{equation-angledeltagammaa}
 \beta\coloneqq\angleLorentzien{\gamma_\azero^+}{\delta}_{\delta(0)}
 =\angleLorentzien{\gamma_\azero^-}{\delta}_{\delta(1)}.
\end{equation*}
The exterior angles at $v_1=\delta(1)$, $v_2=\gamma_\bzero(1)$,
$v_3=\gamma_\bzero(0)$ and $v_4=\delta(0)$
are thus respectively equal to
$\nu_1=-\beta$, $\nu_2=-\alpha$, $\nu_3=\beta$
and $\nu_4=\alpha$
according to the relations \eqref{equation-relationsangle}.
They are indicated in orange in
Figure \ref{figure-rectangleGaussBonnet}.
The Gauss-Bonnet formula
of Proposition \ref{proposition-GaussBonnetBirmanNomizu}
implies then that the area of $P$
vanishes, namely that $P$ has empty interior.
Since $\delta$ and $\gamma_\bzero$ have non-empty interior,
this shows that $\delta=\gamma_\bzero$,
hence that $u=0$ according to \eqref{equation-extremitesdelta}.
Finally $n=u=0$ in both cases
\emph{i.e.} $\mzero_1=\mzero_2$,
which concludes the proof.

% \bibliographystyle{alpha}
% \bibliography{rigiditespectremarquetemps-biblio.bib}
% \section*{Bibliography}
% \addcontentsline{toc}{section}{Bibliographie}
\printbibliography[title={Bibliography}]
% \addtocontents{toc}{\protect\setcounter{tocdepth}{0}}

\end{document}